\documentclass[reqno,10pt]{amsart}  

\usepackage{amsmath,amsfonts,amssymb,amsthm,epsfig,amscd,}
\usepackage[]{fontenc}
\usepackage[latin1]{inputenc}
\usepackage{enumerate}
\usepackage[cmtip,all]{xy}
\usepackage{tabularx,multirow}
\usepackage{color}
\usepackage{url}
\usepackage{color}
\usepackage{tikz-cd}
 \allowdisplaybreaks \numberwithin{equation}{section}

\newtheorem{theorem}{Theorem}[section]
\newtheorem*{theorem*}{Main Theorem}

\newtheorem{proposition}[theorem]{Proposition}
\newtheorem{corollary}[theorem]{Corollary}
\newtheorem{lemma}[theorem]{Lemma}

\newtheorem{claim}[theorem]{Claim}

\theoremstyle{definition}

\theoremstyle{remark}
\newtheorem{remark}[theorem]{Remark}

\newcommand{\Pp}{\mathbb{P}}

\def\geq{{\geqslant}}

\begin{document}

\title[On  Hilbert schemes of curves]{On some components of Hilbert schemes of curves}

\thanks{This collaboration has benefitted of funding from the MIUR Excellence Department Project awarded to the Department of Mathematics,
University of Rome Tor Vergata (CUP: E83-C18000100006) and from the MIUR Excellence Department Project awarded to the Department of Mathematics and Physics, University Roma Tre. Both authors are members of INdAM--GNSAGA}

\author{Flaminio Flamini, Paola Supino}

\noindent
\address{F. Flamini, Dipartimento  di Matematica, Universit\`{a}  degli Studi di Roma ``Tor Vergata", Viale della Ricerca Scientifica 1, 00133 Roma -- Italy}
\email{flamini@mat.uniroma2.it}

\noindent
\address{P. Supino, Dipartimento  di Matematica e Fisica, Universit\`{a} degli Studi ``Roma Tre", Largo S. L. Murialdo 1, 00146
Roma -- Italy}
\email{supino@mat.uniroma3.it}

\begin{abstract} Let $\mathcal{I}_{d,g,R}$ be the union of irreducible components of the
Hilbert scheme whose general points parametrize smooth, irreducible, curves of degree $d$, genus $g$, which are non--degenerate
in the projective space $\mathbb{P}^R$. Under some numerical assumptions on $d$, $g$ and $R$, we construct irreducible components of
$\mathcal{I}_{d,g,R}$ other than the so-called {\em distinguished component}, dominating the moduli space
$\mathcal{M}_g$ of smooth genus--$g$ curves,  which are generically smooth and 
turn out to be of dimension higher than the expected one. The general point of any
such a component corresponds to a curve $X \subset \mathbb{P}^R$ which is a suitable ramified $m$--cover of an irrational curve
$Y \subset \mathbb{P}^{R-1}$, $m \geqslant 2$,   lying in a surface cone over $Y$. The paper extends some of the results
in \cite{CIK17, CIK20}.
\end{abstract}

\keywords{Hilbert scheme of curves, Brill--Noether theory, ruled surfaces, cones, coverings, Gaussian--Wahl maps.}

\subjclass[2010]{Primary 14C05; Secondary 14E20, 14F05, 14J10, 14J26, 14H10}

\maketitle


\section*{Introduction} \label{intro} Projective varieties are distributed in {\em families}, obtained by suitably varying
the coefficients of their defining equations. The study of these families and, in particular, of the properties of their parameter spaces
is a central theme in Algebraic Geometry and sets on  technical tools, like {\em flatness}, {\em base change}, etc.,
as well as on the existence  (due to Grothendieck, with refinements by Mumford) of  the so called {\em Hilbert scheme}, a closed,
projective scheme parametrizing closed projective subschemes with fixed numerical/projective invariants (i.e.\,the\,{\em Hilbert polynomial}), 
and having fundamental {\em universal} properties.

Hilbert schemes have interested several authors over the decades, owing also to deep connections with several other subjects in Algebraic Geometry (cf.\;e.g.\;bibliography in \cite{S06} for an overview).  Indeed, results and techniques in the ``projective domain" of the Hilbert schemes have frequently built bridges towards other topics in Algebraic Geometry, as by improving already known results, as by providing new ones. The interplay between Hilbert schemes of curves in projective spaces and the Brill--Noether theory of line bundles on curves is one of the milestone in Algebraic Geometry (cf.\;e.g.\cite{ACGH,CS,Har82}). The construction of the moduli space $\mathcal M_g$ of smooth, genus--$g$ curves (and its generalizations ${\mathcal M}_{g,n}$ of moduli spaces of smooth, $n$--pointed, genus--$g$ curves), the proof of its irreducibility and the construction of a natural compactification of it deeply rely on the use of Hilbert schemes of curves (c.f.\;e.g. \cite{ACG,DM,Har87}).  Similarly,  together with the {\em Deligne--Mumford compactification} of ${\mathcal M}_g$ in \cite{DM}, the use of Hilbert schemes of curves has been also fundamental in the costruction of suitable compactifications of the {\em universal Picard variety} (cf.\;e.g\;\cite[Theorem,\;p.\;592]{Capo}). 

Besides these examples, the use of Hilbert schemes has been fundamental for several other issues in Algebraic Geometry: unirationality and/or Torelli's type of theorems for cubic hypersurfaces and for prime Fano threefolds of given genus have been proved via the use of Hilbert schemes of lines and planes contained in such varieties (cf.\;e.g.\cite{Fano, ClemGrif, BeauDona, Dona, Vois, FlamSern, IlieMark}). Important connections between Hilbert schemes parametrizing $k$--linear  spaces contained in complete intersections of hyperquadrics and intermediate Jacobians (cf.\;\cite{DonaHype}) are worth to be mentioned, whereas in \cite{CCFM08,CCFM09} the Hilbert schemes of projective scroll surfaces have been related with families of rank--$2$ vector-bundles as well as with moduli spaces of (semi)stable ones. Surjectivity of Gaussian--Wahl maps on curves with general moduli (\cite{CHM88,CLM96}) has deep reflections  both on suitable Hilbert schemes of associated cones and on the extendability of such curves (especially in the $K3$--case). At last, Hilbert schemes parametrizing lines in suitable complete intersections are used either in \cite{BCFS19}, to deduce upper--bounds of minimal gonality of a family of curves  covering    a very--general projective hypersurface of high degree, or in \cite{BCFS20, BCFS20b} to deduce new results concerning either enumerative properties or a certain ``algebraic hyperbolicity" behavior.

In the present paper we focus on  Hilbert schemes of smooth, irreducible projective curves of given degree and genus, the study of which is classical and goes back to Castelnuovo, 
Halphen ({\em Casteluovo bounds} and the {\em gap problem}) and Severi. 

Given non-negative integers $d$, $g$ and $R \geqslant 3$, we denote by $\mathcal{I}_{d,g,R}$ the union of all irreducible components of the
Hilbert scheme  whose general points parametrize smooth, irreducible, non--degenerate curves of degree $d$ and genus $g$ in the projective space
$\mathbb{P}^R$. A component of $\mathcal{I}_{d,g,R}$ is said to be {\em regular} if it is generically smooth and of the {\em expected dimension}, otherwise
it is said to be {\em superabundant} (cf.\;\S\;\ref{Sernesi} for more details).

Under suitable numerical assumptions, involving the so called {\em Brill-Noether number}, it is well--known that $\mathcal{I}_{d,g,R}$ has a unique irreducible component
which dominates  the moduli space ${\mathcal M}_g$ parametrizing (isomorphism classes of) smooth, irreducible genus--$g$ curves (cf.\;\cite{Har82} and \S\;\ref{Sernesi} below). {\color{black}This }is called the {\em distinguished component} of the Hilbert scheme.

In \cite{Se21}, Severi claimed the irreducibility of $\mathcal{I}_{d,g,R}$ when $d \geq g+R$, and this was actually proved by Ein for $R=3,\,4$ in
(cf.\,\cite{E86, E87}); further sufficient conditions on $d$ and $g$ ensuring the irreducibility of some $\mathcal{I}_{d,g,R}$ for $R \geqslant 5$ have been found e.g. in \cite{BF}.  
On the other hand, in several cases there have been also given  examples of additional {\em non--distinguished} components of $\mathcal{I}_{d,g,R}$.
Some of these extra components have been constructed by using either $m$--sheeted covers of $\Pp^1$ (cf.\,e.g.\cite{K94}, \cite{MS89}, etc.), or
by using double covers of irrational curves (cf. e.g. \cite{CIK17}, \cite{CIK20}, etc.) or even by using non--linearly normal
curves in projective space (Harris, 1984 unpublished, see e.g.\;\cite[Ch. IV]{CS}).

In this paper we prove the following:

\vskip 5pt

\begin{theorem*} Let $\gamma \geqslant 10$, $e \geqslant 2 \gamma -1$, $R = e - \gamma + 1$ and $m \geqslant 2$ be integers.
Set $$d := me \;\;\; {\rm and} \;\;\; g := m(\gamma-1)+\frac{m(m-1)}{2}e +1.$$

Then $\mathcal{I}_{d,g,R}$ contains an irreducible component which is generically smooth and {\em superabundant}, having dimension
 $$\lambda_{d,g,R} + \sigma_{d,g,R},$$where
$$\lambda_{d,g,R} := (R+1) me  - (R-3) \left( m(\gamma-1)+\frac{m(m-1)}{2}e  \right)$$ is the {\em expected dimension} of $\mathcal{I}_{d,g,R}$ whereas the positive integer
$$\sigma_{d,g,r} := (R-4) \left[(\gamma-1)(m-1) + 1 + e + \frac{m(m-3)}{2}e\right] + 4 (e+1) + e m (m-5)$$is the {\em superabundance summand} for the dimension of such a component.
\end{theorem*} As additional result, we explicitly describe a general point of the aforementioned superabundant component (cf. Proposition \ref{prop:Flam7e} and
\S\;\ref{superabundant} below). We want to stress that Main Theorem extends some of the results in \cite{CIK17, CIK20} which deal with the case $m=2$.

\vskip 5pt

The paper consists of three sections. In Section \ref{Pre} we remind some generalities concerning Hilbert schemes of curves and associated Brill--Noether theory (cf.\S\,\ref{Sernesi}),
Gaussian--Wahl maps and Hilbert schemes of cones (cf.\S\;\ref{Wahlmap}) and ramified coverings of curves (cf.\;\S\;\ref{Coverings}), which will be used for our analysis. Section \ref{curvescones} deals with the construction of curves $X$ which fill--up an open dense subscheme of the superabundant component of $\mathcal{I}_{d,g,R}$ mentioned in Main Theorem above. Precisely in \S\;\ref{curves} we more generally consider, for any $\gamma \geqslant 1$ and $e \geqslant 2\gamma -1$, curves $Y$  of genus $\gamma$, degree $e$ with {\em general moduli}, which are non--special and projectively normal in $\Pp^{R-1}$ and which fill--up the distinguished component of the related Hilbert scheme ${\mathcal I}_{e,\gamma,R-1}$. Then in \S\;\ref{cones} we consider cones $F = F_Y$ extending  in $\Pp^R$ curves $Y$ as above, we describe abstract resolutions of  cones $F$, together with further cohomological properties (see Proposition \ref{prop:Flam7}), as well as an explicit parametric description 
of the parameter space of such cones, as $Y$ varies in the distinguished component of ${\mathcal I}_{e,\gamma,R-1}$. In \S\;\ref{curvesramif} we construct the desired curves $X$ as curves sitting in cones $F$ as   $m$--sheeted ramified covers $\varphi: X \to Y$, where   the map $\varphi $ is given by the projection
from the vertex of the cone. We prove that such curves $X$ are non--degenerate and linearly normal in $\Pp^R$, we moreover compute their genus $g$ and some other
useful cohomological properties (cf. Proposition \ref{prop:Flam7e}). We also prove Lemma \ref{lem:gradi}, a technical result which deals with a more general situation involving 
projections and ramified covers of possibly reducible, connected, nodal curves and which is needed for a certain inductive procedure used in proving Main Theorem (see Lemma 
\ref{lem:SdgR} and the proof of Claim \ref{cl:induction}). Finally, Section \ref{superabundant} focuses on the proof of Main Theorem, which also involves surjectivity of suitable Gaussian--Wahl maps (cf. proof of Claim \ref{cl:induction}). This explains why in this last section, as well as in Main Theorem, the hypothesis $\gamma \geqslant 10$ is required (cf.\;Proposition \ref{prop:2.1(2.8)Y}).


\subsection*{Notation and terminology} We work throughout over the field $\mathbb{C}$ of complex numbers. All schemes will be endowed with the Zariski topology. By \emph{variety}   we mean an integral algebraic scheme and by \emph{curve} we intend a variety of dimension 1. We say that a property holds for a \emph{general}  point $x$ of a variety $X$ if it holds for any point in a Zariski open non--empty subset of $X$. We will  interchangeably use the terms rank-$r$ vector bundle on a variety $X$ and rank-$r$ locally free sheaf. To ease notation and when no confusion arises, we sometimes identify line bundles with Cartier divisors, interchangeably using additive notation instead of multiplicative notation and tensor products; we moreover denote by $\sim$ the linear equivalence of divisors and by $\equiv$ their numerical equivalence. If $\mathcal P$ is either a parameter space of a flat family of closed subschemes of a variety $X$, as e.g. $\mathcal P$ a Hilbert scheme, or a moduli space parametrizing geometric objects modulo a given equivalence relation, as e.g. the moduli space of smooth genus--$g$ curves, we will denote by $[Y]$ the parameter point (resp., the moduli point) corresponding to the subscheme $Y \subset X$ (resp., associated to the equivalence class of $Y$). For non--reminded terminology, we refer the reader to \cite{H}.


\section{Generalities}\label{Pre}
We briefly recall some generalities and results which will be used in the next sections.


\subsection{Hilbert schemes and Brill-Noether theory of curves}\label{Sernesi}

Let $C$ be a smooth, irreducible, projective curve of genus $g > 0$. Given positive integers $d$ and $r$, the
{\em Brill--Noether locus}, $W^r_d(C) \subseteq {\rm Pic}^d(C)$, when not empty, parametrizes degree--$d$
line bundles $L$ on $C$ such that $h^0(C,L) \geqslant r+1$. Its {\em expected dimension} is given by the
so called  {\em Brill-Noether number}
\begin{equation}\label{eq:rho}
\rho(g,r,d) := g - (r+1) ( g + r - d).
\end{equation} 
{\color{black}It is well-known that if } $C$  {\em   has general moduli} (i.e. when $C$ corresponds to a general point of the moduli space
${\mathcal M}_g$ parametrizing isomorphism classes of smooth, genus--$g$ curves) it is well known that $W^r_d(C)$
is empty if $\rho(g,r,d) <0$, whereas it is generically smooth, of the expected dimension $\rho(g,r,d)$, {\color{black}otherwise.}
Moreover, when $\rho(g,r,d) >0$, $W^r_d(C)$ is also irreducible and for a general $L$ parametrized by $W^r_d(C)$ it is 
$h^0(C,L) = r+1$ (cf.\,\cite[Ch.\,IV,\,V,\,VI]{ACGH}).
 
 Brill-Noether theory of line--bundles on abstract projective curves $C$ is intimately related to the study of Hilbert schemes parametrizing projective
embeddings of such curves. Indeed, assume for simplicity $L \in  W^r_d(C)$ very--ample and such that $h^0(C, L) = r+1$; hence
one has an embedding $C \stackrel{\phi_{|L|}}{\hookrightarrow} \mathbb{P}^r$ induced by the complete linear system $|L|$ determined by $L$, whose image
$Y := \phi_{|L|}(C)$ is a smooth, irreducible curve of degree $d$, genus $g$ which is non--degenerate in $\mathbb{P}^r$. If we denote by $Hilb_{d,g,r}$ the Hilbert scheme  parametrizing
closed subschemes of ${\mathbb P}^r$ with Hilbert polynomial $P(t) = d t + (1-g)$, then $Y$ corresponds to a point of $Hilb_{d,g,r}$. If we denote by
$\mathcal{I}_{d,g,r}$ the union of all irreducible components of $Hilb_{d,g,r}$ whose general points parametrize smooth, irreducible, non--degenerate curves in $\Pp^r$, then $Y$
represents a point $[Y] \in \mathcal{I}_{d,g,r}$. When $[Y]$ is a smooth point of  $\mathcal{I}_{d,g,r}$, then
$Y$ is said to be {\em unobstructed} in ${\mathbb P}^r$.

If $N_{Y/{\mathbb P}^r}$ denotes the {\em normal bundle} of $Y$ in ${\mathbb P}^r$, one has
\begin{equation}\label{eq:dimhilb}
T_{[Y]} (\mathcal{I}_{d,g,r}) \cong H^0(Y, N_{Y/{\mathbb P}^r}) \;\;\;\; {\rm and}\;\;\;\;
\chi (Y, N_{Y/{\mathbb P}^r})  \leqslant \dim_{[Y]}\, \mathcal{I}_{d,g,r} \leqslant h^0(Y, N_{Y/{\mathbb P}^r}),
\end{equation} where the integer $\chi (Y, N_{Y/{\mathbb P}^r})=
h^0(Y, N_{Y/{\mathbb P}^r}) - h^1(Y, N_{Y/{\mathbb P}^r})$ in \eqref{eq:dimhilb} is the so--called {\em expected dimension} of
$\mathcal{I}_{d,g,r}$ at $[Y]$ and  the equality on the right--most--side in  \eqref{eq:dimhilb} holds iff
$Y$ is unobstructed in ${\mathbb P}^r$ (for full details, cf.\,e.g.\,\cite[Cor.\,3.2.7, Thm.\,4.3.4, 4.3.5]{S06}).

The expected dimension of $\mathcal{I}_{d,g,r}$, given by $\chi (Y, N_{Y/{\mathbb P}^r})$, can be easily computed with the use of normal and Euler sequences
for $Y \subset \Pp^r$, and it turns out to be
\begin{equation}\label{eq:expdimHilb}
\lambda_{d,g,r}:= \chi (Y, N_{Y/{\mathbb P}^r}) = (r+1) d - (r-3) (g-1).
\end{equation} A component of $\mathcal{I}_{d,g,r}$ is said to be {\em regular} if it is both reduced (i.e. generically smooth) and of the expected dimension $\lambda_{d,g,r}$; otherwise it is said to be {\em superabundant}.

By above, any component $\mathcal I$ of $\mathcal{I}_{d,g,r}$ has a natural rational map $$\mu_g: {\mathcal I} \dasharrow \mathcal M_g,$$ which simply sends $[Y] \in {\mathcal I} $ general 
to the moduli point $[C] \in \mathcal M_g$ as above. The map $\mu_g$ is called the {\em modular morphism} of $\mathcal I$; with same terminology as in
\cite[Introduction]{S84}, the dimension of ${\rm Im}(\mu_g)$ is called the {\em number of moduli of} $\mathcal I$. 

The expected dimension of ${\rm Im}(\mu_g)$ is ${\rm min} \{3g-3, \, 3g-3 + \rho(g,r,d)\}$, where $\rho(g,r,d)$ as in \eqref{eq:rho}, and it is called the {\em expected number of moduli} of $\mathcal I$.
The expression of the expected number of moduli of $\mathcal I$ is the obvious postulation which comes from the well--known interpretation, in terms of maps between vector bundles on Picard scheme, of the existence of special line bundles on $C$ (cf. \cite[Ch. IV, V, VI]{ACGH}). In this set--up, we remind the following result due to Sernesi:

\begin{theorem}\label{thm:Sernesi} (cf. \cite[Theorem, p.26]{S84}) For any integers $r \geqslant 2,\;d$ and $g$ such that
\[
d \geqslant r+1 \;\;\; {\rm and} \;\;\; d-r \leqslant g  \leqslant \frac{r(d-r)-1}{r-1}
\] there exists a component $\mathcal I$ of ${\mathcal I}_{d,g,r}$ which has the {\em expected number of moduli}. Moreover,
$[Y] \in \mathcal I$ general corresponds to an unobstracted  curve $Y \subset \Pp^r$ such that $h^1(Y, N_{Y/{\mathbb P}^r}) = 0$ and whose embedding in $\Pp^r$ is given by a complete
linear system.
\end{theorem}

\begin{remark}\label{rem:Sernesi} {\rm We want to stress the ``geometric counter--part" of the
numerical hypotheses appearing in Theorem \ref{thm:Sernesi}. For $Y$ as in Theorem \ref{thm:Sernesi}, let $(C,L)$ be the pair consisting 
of a smooth, irreducible, abstract projective curve $C$ of genus $g$ and of $L \in {\rm Pic}^d(C)$ such that 
$Y = \phi_{|L|}(C)$. Then, condition $d \geqslant r+1$ simply means that the curve $Y$ is of positive genus and non--degenerate in
$\Pp^r$ whereas $d-r \leqslant g$, i.e. $g + r - d \geqslant 0$, simply decodes by Riemann--Roch the condition that the {\em index of speciality} $i(L) := g+r-d$
of $L$ is non--negative. At last, the condition $g \leqslant \frac{r(d-r)-1}{r-1}$ reads $g - rg + rd - r^2 - 1 \geqslant 0$ which is nothing but
$\rho(g,r,d) + (r+g-d) = \rho(g,r,d) + i(L) \geqslant 1$, i.e. it is a {\em ``Brill-Noether type"} condition on the pair $(C,L)$.}
\end{remark}

It is well known (cf.\,e.g. \cite[p.\,70]{Har82}) that, when $\rho(g,r,d) \geqslant 0$, $\mathcal{I}_{d,g,r}$ has a {\em unique component} with a dominant modular morphism $\mu_g$, i.e. dominating $\mathcal M_g$; thus such a component has maximal number of moduli $3g-3$. It is called the {\em distinguished component} of $\mathcal{I}_{d,g,r}$ and, in the sequel, we will denote it by $\widehat{\mathcal{I}_{d,g,r}}$ or simply by $\widehat{\mathcal{I}}$, if no confusion arises. As a direct consequence of the uniqueness of $\widehat{\mathcal{I}}$
and of Theorem \ref{thm:Sernesi}, one has:

\begin{corollary}\label{cor:Sernesi}  For any integers $r \geqslant 2,\;d$ and $g$ such that
\[
d \geqslant r+1 \;\;\; {\rm and} \;\;\; d-r \leqslant g  \leqslant \frac{(r+1)(d-r)-1}{r}
\] the distinguished component $\widehat{\mathcal{I}}$ of $\mathcal{I}_{d,g,r}$ is not empty. Its general point $[Y]$ corresponds to an unobstructed  curve $Y$ in $\Pp^r$ with $h^1(Y, N_{Y/{\mathbb P}^r}) = 0$ and whose embedding in $\Pp^r$ is given by a complete linear system. Furthermore $\widehat{\mathcal{I}}$ is regular, i.e. generically smooth and of the expected dimension
$\lambda_{d,g,r}$.
\end{corollary}
\begin{proof} The condition $g  \leqslant \frac{(r+1)(d-r) -1}{r}$ is equivalent to $\rho(g,r,d) \geqslant 1$. Thus we conclude by applying 
Theorem \ref{thm:Sernesi}, taking into account what discuss in Remark \ref{rem:Sernesi}, and by applying \cite[p.\,70]{Har82} and \eqref{eq:dimhilb}, as the condition 
$h^1(Y, N_{Y/{\mathbb P}^r}) =0$ implies both the non--obstructedness of $Y$ in $\Pp^r$ and the regularity of  $\widehat{\mathcal{I}}$.
\end{proof}

In \cite{Se21}, Severi claimed the irreducibility of $\mathcal{I}_{d,g,r}$ when $d \geq g+r$. Severi's claim was proved by Ein for $r=3,\,4$ in
(cf.\,\cite{E86, E87}); further sufficient conditions on $d$ and $g$ ensuring the irreducibility of some $\mathcal{I}_{d,g,R}$ for $R \geqslant 5$ have been found e.g. in \cite{BF}. On the other hand, in several cases there have been also given examples of additional {\em non--distinguished} components of $\mathcal{I}_{d,g,r}$, even in the range $\rho(g,d,r) \geqslant 0$. Some of these extra components have been constructed by using either $m$--sheeted covers of $\Pp^1$ (cf.\,e.g.\cite{K94}, \cite{MS89}, etc.), or
by using double covers of irrational curves (cf. e.g. \cite{CIK17}, \cite{CIK20}, etc.) or even by using non--linearly normal
curves in projective space (the latter approach is contained in a series of examples due to Harris, 1984 unpublished, fully described in e.g. \cite[Ch. IV]{CS}).
In some cases, these extra components have been also proved to be regular
(cf.\;e.g.\;\cite[Ch. IV]{CS}, \cite{CIK20}).

\vskip 15pt


\subsection{Gaussian-Wahl maps and cones}\label{Wahlmap} Let $C$ be a smooth, irreducible projective curve of positive genus $g$ and $L$ be a very--ample line bundle of degree $d$ 
on $C$. Set $Y \subset \Pp^r$ the embedding of $C$ by the complete linear system $|L|$. Let $F_Y$ (equiv., $F_{C,L}$) denote the cone in $\Pp^{r+1}$ over $Y$ with vertex at a point $v \in \Pp^{r+1} \setminus \Pp^r$ (if no confusion arises, in the sequel we simply set $F$).

Fundamental properties of such cones are related to the so--called  {\em Gaussian--Wahl maps} (cf.\,e.g. \cite{CHM88,CLM96,Wa90}), as we will briefly remind. If $\omega_C$ denotes the {\em canonical bundle} of $C$, one sets
$$R(\omega_C, L) := {\rm Ker} \left[H^0(C, \omega_C) \otimes H^0(C,L) \longrightarrow H^0(C, \omega_C \otimes L)\right],$$where the previous map is a
natural multiplication map among global sections. One can consider the map
\begin{equation}\label{eq:Gaussian}
\Phi_{\omega_C, L} : R(\omega_C, L) \to H^0(\omega_C^{\otimes 2} \otimes L),
\end{equation}defined locally by $\Phi_{\omega_C, L}(s \otimes t) := s\,dt - t\,ds$, which is called the {\em Gaussian--Wahl map}. As customary,
one sets
\begin{equation}\label{eq:cork}
\gamma_{C, L} := {\rm cork} (\Phi_{\omega_C, L} ) = \dim \, {\rm Coker} (\Phi_{\omega_C, L} ).
\end{equation}

For reader's convenience we will remind here statement of \cite[Prop.\,2.1]{CLM96}, limiting ourselves to its $(2.8)$--part, which will be used in Section 
\ref{superabundant}; indeed, the full statement
of  \cite[Prop.\,2.1]{CLM96} is quite long, with many exceptions and dwells also on curves with low
genus whereas Section \ref{superabundant} will focus on curves of genus at least $10$. 

\begin{proposition}\label{prop:CLM2.1} (cf. \cite[Proposition 2.1--(2.8)]{CLM96}) Let $ g \geqslant 6$ be an integer. Assume that $C$ is a smooth, projective curve of genus $g$ 
with general moduli and that $L \in {\rm Pic}^d(C)$ is general. Then, $\gamma_{C, L} = 0$ (i.e. $\Phi_{\omega_C, L}$ is surjective) if
\[
d \geqslant \left\{
\begin{array}{cc}
g + 12 & {\rm for} \;\; 6 \leqslant g \leqslant 8\\
g+9 & {\rm for} \;\; g \geqslant 9
\end{array}
\right..
\]
\end{proposition}

Gaussian--Wahl maps can be used to compute the dimension of the tangent space to the Hilbert scheme of surfaces in $\Pp^{r+1}$ at points representing cones $F$ as above
(cf.\,e.g. \cite{CLM96}). Indeed, let $\mathcal W$ be any irreducible component of the Hilbert scheme of curves 
${\mathcal I}_{d,g,r}$ and  let
$\mathcal{H}(\mathcal W)$ be the variety which parametrizes the family of cones $F \subset \Pp^{r+1}$  over curves $Y \subset \Pp^r$ representing
points in $\mathcal W$. Then, one has:

\begin{proposition}\label{prop:CLM2.18} (cf. \cite[Cor.\,2.20--(c), Prop.\,2.12--(2.13) and (2.15)]{CLM96}) Set notation and conditions as in Proposition
\ref{prop:CLM2.1}. Let $r = d-g$, $Y\subset \Pp^r$ and $\mathcal W$ be any component of ${\mathcal I}_{d,g,r}$ s.t.
$[Y] \in \mathcal W$ is general.  Then:

\noindent
(i) The Gaussian--Wahl map $\Phi_{\omega_Y, {\mathcal O}_Y(1)}$ is surjective, i.e. $\gamma_{Y, {\mathcal O}_Y(1)} = 0$.

\noindent
(ii) $h^0(Y, N_{Y/\Pp^r} \otimes {\mathcal O}_Y(-1)) = r+1$.

\noindent
(iii) $h^0(Y, N_{Y/\Pp^r} \otimes {\mathcal O}_Y(-j)) = 0$, for any $j \geqslant 2$.

\noindent
(iv) $\mathcal{H}(\mathcal W)$ is a generically smooth component of the Hilbert scheme parametrizing surfaces of degree
$d$ and sectional genus $g$ in $\Pp^{r+1}$.  Moreover,
\begin{equation}\label{eq:dimHW}
\dim \, \mathcal{H}(\mathcal W) = (r+1) (d+1) - (r-3) (g-1) = \lambda_{d,g,r} + (r+1)
\end{equation} and it is generically smooth, i.e. for $[F] \in \mathcal{H}(\mathcal W)$ general the associated cone $F = F_Y$ is unobstructed in $\Pp^{r+1}$.
\end{proposition}

\vskip 15pt

\subsection{Ramified coverings of curves}\label{Coverings}  Let $Y$ be a scheme. A morphism $\varphi: X \to Y$ 
is called a \emph{covering map of degree} $m$ if $\varphi_*\mathcal{O}_X$ is a locally free ${\mathcal O}_Y$--sheaf of rank $m$. A map $\varphi$ is a covering map (or simply {\em a cover}) if and only if it is finite and flat. In particular, if $Y$ is smooth and irreducible and $X$ is Cohen--Macaulay, then every finite, surjective morphism $\varphi: X \to Y$ is
a covering map (cf.\;e.g.\;\cite[p.\;1361]{CE96}).  

When $\varphi: X \to Y$ is a covering map of degree $m$, one has a natural exact sequence
\[
 0 \to \mathcal{O}_Y \xrightarrow{\varphi^{\sharp}} \varphi_{\ast} \mathcal{O}_X \to \mathcal{T}^{\vee}_{\varphi} \to 0 \, ,
\]where $\mathcal{T}^{\vee}_{\varphi} := Coker (\varphi^{\sharp})$ is the so-called \emph{Tschirnhausen bundle} associated to the covering map $\varphi$, which is of rank $m-1$ on $Y$. 
Since $\textrm{Char}\; (\mathbb{C}) = 0$, the trace map $\textrm{tr}\colon\varphi_*\mathcal{O}_X\rightarrow\mathcal{O}_Y$ gives rise to a splitting of the previous exact sequence, so that one has
$\varphi_*\mathcal{O}_X=\mathcal{O}_Y\oplus \mathcal{T}^{\vee}_{\varphi}$ (cf.\,e.g. \cite{CE96, CIK20, DP15, DP16}).

If $X$ and $Y$ are in particular smooth, irreducible curves and $\varphi:X \to Y$ is a covering map of degree $m$, according to \cite[Ex.\, IV.2.6--(d), p. 306]{H}, the {\em branch divisor} $B_{\varphi}$ of $\varphi$ is such that
\begin{equation}\label{eq:branch}
\left(\bigwedge^m (\varphi_* \mathcal{O}_X)\right)^{\otimes 2} \cong \mathcal{O}_Y(-B_{\varphi}).
\end{equation} If moreover $X$ (resp., $Y$) has genus $g$ (resp., $\gamma$) then one has
$\deg B_{\varphi} = - 2 \deg \left(\bigwedge^m (\varphi_* \mathcal{O}_X)\right) = 2(g - 1) - 2m (\gamma - 1)$.
As for the {\em ramification divisor} $R_{\varphi}$ such that $\varphi(R_{\varphi})=B_{\varphi}$, the Riemann--Hurwitz formula gives
\begin{equation}\label{eq:RiemannH}
\omega_X={\varphi}^*(\omega_Y) + \mathcal{O}_X(R_{\varphi}).
\end{equation}

In this set--up, we recall the {\em pinching construction} described in \cite[\S\,3.1]{DP16}. Let $\varphi :X \to Y$ be a degree--$m$ covering map between smooth irreducible curves $X$ and $Y$. Let $Z$ be the reduced, reducible nodal curves $$Z := X \cup Y,$$where $X$ and $Y$ are {\em attached nodally}
at $\delta$ distinct points as follows: let $y_i \in Y$ and $x_i \in X$ be points such that $\varphi(x_i )= y_i$, $1 \leqslant i \leqslant \delta$. Set
$D := \sum_{i=1}^{\delta} y_i$, $\mathcal{O}_D$ the structural sheaf of $D$ and $\mathcal{J}$ the kernel of the map
\[
 \varphi_* \mathcal{O}_X \oplus \mathcal{O}_Y \to  \mathcal{O}_D,
 \]
defined around any $y_i$'s as
\[
 (f,g) \mapsto f(x_i) - g(y_i), \;\; \forall\; \; 1 \leqslant i \leqslant \delta.
 \]Then $\mathcal{J} \subset \varphi_* \mathcal{O}_X \oplus \mathcal{O}_Y$ is an $\mathcal{O}_Y$-subalgebra of $\varphi_* \mathcal{O}_X \oplus \mathcal{O}_Y$
and $\mathrm{Spec}_Y (\mathcal{J}) = Z = X \cup Y$. $D$ is called the {\em set of nodes} of $Z$.  Let $\psi : Z \to Y$ be the natural induced finite and surjective map. Since $Y$ is smooth, irreducible and $Z$ is l.c.i. (so in particular Cohen--Macaulay),  from what reminded above the map 
$\psi$ is a covering map of degree $m+1$. In this set--up, one has the following:

\begin{proposition}\label{prop:split} (cf. \cite[Lemma\,3.2]{DP16}) Let $\varphi :X \to Y$ be a covering map of degree--$m$ between smooth irreducible curves $X$ and $Y$. Let 
$\psi : Z \to Y$ be the covering map of degree $m+1$ induced by the pinching construction whose set of nodes is $D$. Then, the following 
exact sequence of vector bundles on $Y$
 \[
  0 \to \mathcal{T}_{\varphi} \to \mathcal{T}_{\psi} \to \mathcal{O}_Y(D) \to 0
  \]holds, where $\mathcal{T}^{\vee}_{\varphi}$ and $\mathcal{T}^{\vee}_{\psi}$ are the {\em Tschirnhausen bundles} associated to the covering maps $\varphi$ and $\psi$, respectively.
\end{proposition}

\vskip 15pt


\section{Curves and cones}\label{curvescones} In this section we first construct families of non--special curves $Y$ of any positive genus $\gamma$ and of degree $e\geqslant 2 \gamma -1$ 
in a projective space, which turn out to fill--up the distinguished component $\widehat{\mathcal{I}}$ of the related Hilbert scheme (cf.\;\S\;\ref{curves}). After that, we deal with the family ${\mathcal H} (\widehat{\mathcal{I}})$, as in \S\,\ref{Wahlmap}, which parametrizes
cones extending curves in $\widehat{\mathcal{I}}$, i.e. cones having curves in $\widehat{\mathcal{I}}$ as hyperplane sections. We describe an abstract resolution of a general point of ${\mathcal H} (\widehat{\mathcal{I}})$, and compute $\dim\; {\mathcal H} (\widehat{\mathcal{I}})$ via an explicit parametric description (cf.\;\S\;\ref{cones}). To conclude the section, for cones $F$ parametrized by ${\mathcal H} (\widehat{\mathcal{I}})$ we construct smooth, irreducible curves $X\subset F$, of suitable degree $d$ and genus $g$, which turn out to be $m$--sheeted ramified covers of
curves $Y$ varying in the distinguished component $\widehat{\mathcal{I}}$ (cf.\;\S\;\ref{curvesramif}). 


\subsection{Curves in distinguished components}\label{curves} Let $\gamma >0 $ and $e \geqslant 2 \gamma -1$ be integers. Let $C$ be a smooth, irreducible, projective curve of genus $\gamma$ and let ${\mathcal O}_C(E) \in {\rm Pic}^e(C)$ be a general line bundle. Thus, ${\mathcal O}_C(E)$ is very--ample and non--special (i.e. $h^1(C, {\mathcal O}_C(E))=0$). By Riemann-Roch, we set
\begin{equation}\label{eq:R}
R:= h^0(C, {\mathcal O}_C(E)) = e - \gamma +1,
\end{equation}so that $|{\mathcal O}_C(E)|$ defines an embedding
$C \stackrel{\phi_{|E|}}{\hookrightarrow} \Pp^{R-1}$, whose image we denote from now on by $Y:= \phi_{|E|}(C)$. Taking into account \cite[Thm.\;1]{GL86}, we therefore have:
\begin{equation}\label{eq:Y}
Y \; \mbox{is a smooth, projective curve of genus $\gamma >0$, degree $e \geqslant 2\gamma-1$, which is projectively normal in $\Pp^{R-1}$. }
\end{equation} As in \S\,\ref{Sernesi}, one has in particular that $[Y] \in {\mathcal I}_{e, \gamma, R-1}$.

If we let vary $[C] \in \mathcal M_{\gamma}$ and, for any such
$C$, we let ${\mathcal O}_C(E)$ vary in ${\rm Pic}^e(C)$, the next result shows that the  corresponding curves $Y \subset \Pp^{R-1}$ fill--up the distinguished component $\widehat{\mathcal{I}} := \widehat{\mathcal{I}_{e,\gamma,R-1}}$ which also turns out to be regular.

\begin{proposition}\label{prop:SernesiY} Let $\gamma >0$ and $e \geqslant 2 \gamma -1$ be integers. Let $C$ be a smooth, projective curve of genus $\gamma$ with general moduli, and let
${\mathcal O}_C(E) \in {\rm Pic}^e(C)$ be a general line bundle. Let $Y:= \phi_{|E|}(C) \subset \Pp^{R-1}$, where $R = e - \gamma +1$.
Then, $Y$ is a smooth, irreducible curve of degree $e$ and genus $\gamma$ which is projectively normal in $\Pp^{R-1}$, as an embedding of $C$ via the complete linear system $|E|$, and such that $h^1(Y, N_{Y/\Pp^{R-1}}) =0$. It corresponds to a general point of the
{\em distinguished component} $\widehat{\mathcal{I}} := \widehat{\mathcal{I}_{e,\gamma,R-1}}$ of the Hilbert scheme ${\mathcal I}_{e, \gamma, R-1}$, which is regular of dimension
\begin{equation}\label{eq:lambdae}
\dim \, \widehat{\mathcal{I}} = \lambda_{e,\gamma,R-1} = R e - (R-4) (\gamma -1).
\end{equation}
\end{proposition}

\begin{proof} Numerical assumptions and \cite[Thm.\;1]{GL86} imply that $E$ is very--ample, non--special and that $Y \subset \Pp^{R-1}$ is projectively normal, the equality
$R = e - \gamma +1$ simply following by the non--speciality of $E$ and by Riemann-Roch.

Under our assumptions, numerical hypotheses in Corollary \ref{cor:Sernesi} hold true. Indeed, as explained in Remark \ref{rem:Sernesi} we have the following: since 
$Y \subset \Pp^{R-1}$ is non--degenerate and of positive genus $\gamma$, then condition $e \geqslant R$ is certainly satisfied; concerning condition
$\gamma \geqslant e - (R-1)$, i.e. $i(E) \geqslant 0$,  it certainly holds from the non--speciality of $E$; at last, non--speciality of $E$ gives 
$\rho(\gamma, R-1, e) + i(E) = \rho(\gamma, R-1, e) = \gamma \geqslant 1$, therefore $\gamma \leqslant \frac{R(e - (R-1)) -1}{R-1}$ as in Corollary \ref{cor:Sernesi} 
certainly holds (cf. Remark \ref{rem:Sernesi}).

Thus, by Corollary \ref{cor:Sernesi}, $[Y]$ corresponds to a point in  the distinguished component $\widehat{\mathcal{I}} $ of $\mathcal{I}_{e,\gamma, R-1}$ and is such 
that $h^1(Y, N_{Y/\Pp^{R-1}}) =0$, i.e. $\widehat{\mathcal{I}}$ is generically smooth and of the expected dimension
$\lambda_{e,\gamma,R-1}$ which equals $R e - (R-4) (\gamma -1)$, as it follows from \eqref{eq:expdimHilb}.
\end{proof}

\vskip 15pt


\subsection{Cones extending curves in $\widehat{\mathcal{I}}$}\label{cones}

With notation as in \S\,\ref{Wahlmap}, here we will deal with the family of cones ${\mathcal H} (\widehat{\mathcal{I}})$, where $\widehat{\mathcal{I}} = \widehat{\mathcal{I}_{e,\gamma,R-1}}$
is the distinguished component in Proposition \ref{prop:SernesiY} above.  For $[Y] \in \widehat{\mathcal{I}}$ general, we will denote by $F:= F_Y \subset \Pp^R$ a cone over $Y$ with general vertex $v \in \Pp^R \setminus \Pp^{R-1}$. {\color{black}In order to } describe suitable smooth, abstract resolution of the cones $F$,{\color{black} we recall the following general facts.}

Let $C$ be a smooth, irreducible projective curve of genus $\gamma > 0$ and let ${\mathcal O}_C(E) \in {\rm Pic}^e(C)$ be a general line bundle of degree $e \geqslant 2\gamma -1$.
Consider the rank--two, normalized vector bundle $\mathfrak{F} := \mathcal{O}_C\oplus  \mathcal{O}_C(-E)$ on $C$ and let $S := \mathbb{P}(\mathfrak{F})=
\mathrm{Proj}_C (\mathrm{Sym} (\mathfrak{F}))$ the associated geometrically ruled surface on $C$. One has the structural morphism
$\rho: S \to  C$ such that $\rho^{-1} (p) = f_p$, for any $p \in C$, where $f_p \cong \Pp^1$ denotes the fiber of the ruling of $S$ over the point $p \in C$. A general fiber of the ruling of $S$ will be simply denoted by $f$.

$S$ is endowed with two natural sections, $C_0$ and $C_1$, both isomorphic to $C$, and
such that $C_0\cdot C_1 = 0$, $C_0^2 = - C_1^2 = -e$. The section $C_0$ (resp., $C_1$) corresponds to the exact sequence
\[
0 \to \mathcal{O}_C \to \mathfrak{F} \to \mathcal{O}_C (-E) \to 0 \;\;\; ({\rm resp.,}\;\;\;0 \to \mathcal{O}_C (-E) \to \mathfrak{F} \to \mathcal{O}_C \to 0).
\] Moreover, one has ${\rm Pic}(S) \cong \mathbb{Z}[\mathcal{O}_S(C_0)] \oplus \rho^*({\rm Pic}(C))$ and ${\rm Num} (S) \cong \mathbb{Z} \oplus \mathbb{Z}$ 
(cf.\,e.g. \cite[V.2]{H}). To ease notation, for any $D \in {\rm Div}(C)$, we will simply set $\rho^*(D) := D\,f$.

If $K_S$ (resp., $K_C$) denotes a canonical divisor
of $S$ (resp., of $C$), one has (cf.\,e.g. \cite[V.2]{H}):
\begin{equation}\label{eq:KS}
C_1 \sim C_0 + E \;f \;\;\; {\rm and} \;\;\; K_S \sim - 2 C_0 + (K_C-E)\,f.
\end{equation}

\begin{proposition}\label{prop:Flam7} Let $C$ be a smooth, irreducible projective curve of genus $\gamma > 0$ and ${\mathcal O}_C(E) \in {\rm Pic}^e(C)$ be a general line bundle of degree $e \geqslant 2\gamma -1$. Consider the normalized, rank--two vector bundle $\mathfrak{F} := \mathcal{O}_C\oplus  \mathcal{O}_C(-E)$ on $C$ and let $S := \mathbb{P}(\mathfrak{F})$, together with the natural sections $C_0$ and $C_1$, where $C_1 \sim C_0 + E \;f, \; C_0\cdot C_1 = 0, \; C_0^2 = - C_1^2 = - e$. Then:

\noindent
(i) The linear system $|{\mathcal O_S}(C_1)|$ is base--point--free and not composed with a pencil. It induces a morphism
$$\Psi := \Psi_{|{\mathcal O}_S(C_1)|}: S \to \Pp^R,$$where $R = e - \gamma +1$.

\noindent
(ii) $\Psi$ is an isomorphism, outside the section $C_0 \subset S$, onto its image $F := \Psi(S) \subset \Pp^R$, whereas it contracts $C_0$ at a point
$v \in \Pp^R$.

\noindent
(iii) $F$ is a cone of vertex $v$ over $Y := \Psi(C_1) \cong C$, where
$Y \subset \Pp^{R-1}$ is a hyperplane section of $F$ not passing through $v$, which is smooth, irreducible, non--degenerate,
of degree $e$, genus $\gamma$ and it is also projectively normal in $\Pp^{R-1}$.

\noindent
(iv) The cone $F \subset \Pp^R$ is projectively normal, of degree $\deg F = e$, of sectional genus and speciality  $\gamma$. In particular,
$h^0(F,{\mathcal O}_F(1)) = R+1 = e - \gamma +2$ and $h^1(F, {\mathcal O}_F(1)) = \gamma$.

\noindent
(v) For any $m \geqslant 2$, one has $h^0(F, \mathcal{O}_F(m)) = \frac{m(m+1)}{2}e -m(\gamma-1) +1$.

\end{proposition}
\begin{proof}
 (i) From \cite[Ex.\;V.2.11 (a), p.\;386]{H} one deduces that $|{\mathcal O}_S(C_1)|$ is base--point--free and not composed with a pencil. Therefore
$\Psi$ is a morphism and its image is a surface. Now, from \eqref{eq:KS}, we have
$$h^0(S, {\mathcal O}_S(C_1)) = h^0(S, {\mathcal O}_S(C_0 + Ef)) = h^0(C, {\mathcal O}_C (E)) + h^0(C, {\mathcal O}_C) = (e-\gamma +1) +1 = R +1,$$
where the second equality follows from Leray's isomorphism, projection formula and the fact that
$$\rho_*({\mathcal O}_S(C_0 + Ef)) = \rho_*({\mathcal O}_S(C_0 ) \otimes {\mathcal O}_S( Ef)) = {\mathfrak F} \otimes {\mathcal O}_C(E) =
\left(\mathcal{O}_C\oplus  \mathcal{O}_C(-E)\right) \otimes {\mathcal O}_C(E) = \mathcal{O}_C (E) \oplus  \mathcal{O}_C,$$whereas the third and the last
equality follow, respectively, from the fact that $E$ is non--special of degree $e$ on $C$ of genus $\gamma$ and from \eqref{eq:R}.

\vskip 2pt

\noindent
(ii) Since $E$ is very--ample on $C$, \cite[Prop.\;23]{FP05} implies that the morphism $\Psi$ is an isomorphism onto its image $F$ outside
the section $C_0$ of $S$. On the other hand, $C_1 \cdot C_0 = (C_0 + E f) \cdot C_0 = -e +e = 0$, i.e. $C_1$ contracts the section $C_0$ at a point $v \in
\Pp^R$ which is off  $Y = \Psi(C_1)$, the isomorphic image of the section $C_1 \cong C$.

\vskip 2pt

\noindent
(iii) All the fibers of the ruling of $S$ are embedded as lines, as $C_1 \cdot f = 1$. Since $C_0$ is contracted to a point $v$ and since $C_0 \cdot f = 1$, for any fiber $f$ of the ruling of $S$, it follows that any line $\ell := \Psi(f)$ passes through $v$; thus $F = \Psi(S)$ is a cone over $Y$, of vertex $v \in \Pp^R$. From the isomorphism $C_1 \cong C$, one also
deduces $\Psi_{|_{C_1}} \cong \phi_{|\mathcal{O}_C(E)|}$, as the following diagram summarizes:
\begin{equation}\label{eq:diagPaola}
\begin{tikzcd}
S \arrow[r, "\Psi_{|\mathcal{O}_S(C_1)|}"] \arrow[d, "\rho"] & F \subset \mathbb{P}^R \arrow[d,dashed,"\pi_v"]\\
C \;\;\;\; \arrow[r, "\phi_{|\mathcal{O}_C(E)|}"]& \;\;\;\;\; Y\subset \mathbb{P}^{R-1},
\end{tikzcd}
\end{equation} where $\pi_v$ denotes the projection from the vertex point $v$. Since ${\mathcal O}_F(1)$ is induced by ${\mathcal O}_S(C_1)$, it is clear that $Y$ is a hyperplane section of $F$ so $Y \subset \Pp^{R-1}$ is of degree $e$, genus $\gamma$ and it is projectively normal in $\Pp^{R-1}$, as it follows from \cite[Thm.\;1]{GL86}.

\vskip 2pt

\noindent
(iv) One has $\deg\,F = Y^2 = C_1^2 = e$; moreover, since $Y \cong C$ is a hyperplane section, then $F$ has sectional genus $\gamma$. Now,
$h^0(F, {\mathcal O}_F(1)) = h^0(S, {\mathcal O}_S(C_1)) = R+1 = e - \gamma + 2$, as computed in (i); whereas $h^1(F, {\mathcal O}_F(1)) = h^1(S, {\mathcal O}_S(C_1))$ so, from the exact sequence
$$0 \to {\mathcal O}_S \to {\mathcal O}_S(C_1) \stackrel{r_{C_1}}{\longrightarrow} {\mathcal O}_{C_1}(C_1) \cong {\mathcal O}_C(E) \to 0$$one gets
$$h^0(S, {\mathcal O}_S(C_1) ) = R+1,\;\,\;  h^0(C_1, {\mathcal O}_{C_1}(C_1)) = h^0(C, {\mathcal O}_C(E)) = R, \;\;\; h^1(C_1, {\mathcal O}_{C_1}(C_1)) = h^1(C, {\mathcal O}_C(E)) =0,$$as $E$
is non--special. By Leray's isomorphism and projection formula
one  also gets $h^1(S, {\mathcal O}_S ) = h^1(C, {\mathcal O}_C) = \gamma$. Thus the map $H^0(r_{C_1})$, induced in cohomology by the map $r_{C_1}$, is surjective hence, from the above 
exact sequence, one gets $h^1(S, {\mathcal O}_S(C_1)) = \gamma$.

At last, since $F$ has general hyperplane section which is a projectively normal curve in $\Pp^{R-1}$, it follows
that $F$ is projectively normal in $\Pp^R$ (cf.\,e.g.\,\cite[Proof of Lemma 5.7, Rem. 5.8]{CCFM08}.

\vskip 2pt

\noindent
(v) By the very definition of $F$, one has $h^0(F, \mathcal{O}_F(m)) = h^0(S,\mathcal{O}_S(m C_1))$. From \cite[Ex. III.8.3, p. 253]{H} it follows that
\begin{equation*}
h^0(S,\mathcal{O}_S(m C_1))=h^0(S,\mathcal{O}_S(m C_0 + mEf))=\sum_{k=0}^m h^0(C, \mathcal{O}_C ((m-k)E)) = \sum_{j=0}^m h^0(C, \mathcal{O}_C (j E)).
\end{equation*}Since $E$ is non--special on $C$, so it is any divisor $j E$, for any integer $1 \leqslant j \leqslant m$. Therefore, by Riemann--Roch on $C$, one has
$$h^0(F, \mathcal{O}_F(m)) = \sum_{j=0}^m h^0(C, \mathcal{O}_C (j E)) =  1 + \left[1 +2+3 + \ldots + (m-1) +m\right] e - m\gamma + m = \frac{m(m+1)}{2}e -m(\gamma-1) + 1,$$
as stated.

\end{proof}

Proposition \ref{prop:Flam7} allows to give an explicit {\em parametric description} of the family of cones
${\mathcal H} (\widehat{\mathcal{I}})$, where $\widehat{\mathcal{I}} := \widehat{\mathcal{I}_{e,\gamma,R-1}}$ is
the distinguished component of the Hilbert scheme $\mathcal{I}_{e,\gamma,R-1}$ as in Proposition \ref{prop:SernesiY}. For reader's convenience, we first report here a special case of \cite[Lemma\,6.3]{CCFM09}, which is needed for the parametric description of ${\mathcal H} (\widehat{\mathcal{I}})$.

\begin{lemma}\label{lem:step2} With notation and assumptions as in Proposition \ref{prop:Flam7}, assume further
that ${\rm Aut}(C) = \{Id\}$ (this, in particular, happens when e.g. $C$ has general moduli). Let $G_F \subset {\rm PGL} (R +1, \mathbb{C})$ denote the {\em projective stabilizer} of $F$, i.e.
the sub-group of projectivities of $\Pp^{R}$ which fix $F$ as a cone.

Then $G_F \cong {\rm Aut}(S)$ and $\dim \, G_S  = h^0(C, {\mathcal O}_C(E)) + 1 = R+1 = e - \gamma +2$.
\end{lemma}
\begin{proof} There is an obvious inclusion $G_F \hookrightarrow {\rm Aut}(S)$; we want to show that this is actually a group isomorphism.

Let $\sigma \in {\rm Aut}(S)$ be any automorphism of $S$. Since $C_0$ is the unique section of $S$ with negative self--intersection, then
$\sigma (C_0)= C_0$, i.e. $\sigma$ induces an automorphism of $C_0 \cong C$. Assumption ${\rm Aut}(C) = \{Id\}$ implies that 
$\sigma$ fixes $C_0$ pointwise. Now, from the fact that $C_1 \sim C_0 + E\,f$, it follows that 
$\sigma^*(C_1) \sim \sigma^*(C_0) + \sigma^*(E\,f) = C_0 + E\,f \sim C_1$. Therefore, since $|C_1|$ corresponds to the hyperplane linear system of $F = \Psi(S)$,
one deduces that any automorphism $\sigma \in {\rm Aut}(S)$ is induced by a projective transformation of $F$.

The rest of the proof directly follows from cases \cite[Theorem 2--(2) and (3)]{Ma2} and from \cite[Lemma 6]{Ma2}: indeed condition
${\rm Aut}(C) = \{Id\}$ implies that ${\rm Aut}(S) \cong {\rm Aut}_C(S)$; furthermore, since $C_0$ is the unique section of negative self--intersection on $S$,
$\dim\,G_S  = h^0(C, {\mathcal O}_C(E)) + 1$ follows by using the description of ${\rm Aut}_C(S)$ in \cite[Theorem 2]{Ma2}.
\end{proof}

With the use of Proposition \ref{prop:Flam7} and of Lemma \ref{lem:step2}, one can explicitly describe the family of cones
${\mathcal H} (\widehat{\mathcal{I}})$, where $\widehat{\mathcal{I}}$ the distinguished component in Proposition \ref{prop:SernesiY}, and also 
compute its dimension.

\vskip 15pt

\noindent
{\bf Parametric description of ${\mathcal H} (\widehat{\mathcal{I}})$}: letting $[C]$ vary in ${\mathcal M}_{\gamma}$ and, for any such 
$C$, letting ${\mathcal O}_C(E) $ vary in ${\rm Pic}^e(C)$, cones $F$ arising as in Proposition \ref{prop:Flam7} 
fill-up the component ${\mathcal H} (\widehat{\mathcal{I}})$, which depends on the following parameters:
\begin{itemize}
\item $3\gamma -3$, since $[C]$ varies in ${\mathcal M}_{\gamma}$, plus
\item $\gamma$, which are the parameters on which ${\mathcal O}_C(E) \in {\rm Pic}^e(C)$ depends, plus
\item $(R+1)^2 -1 = \dim \; {\rm PGL}(R+1, \mathbb{C}),$ minus
\item $\dim\,G_F$, which is the dimension of the projectivities of $\Pp^{R}$ fixing a
general $F$ arising from this construction.
\end{itemize} From Lemma \ref{lem:step2} it follows that $\dim\,G_F = R+1$, so 
\begin{equation}\label{eq:dimZbis}
\dim\;{\mathcal H} (\widehat{\mathcal{I}}) = 4\gamma - 3 + (R+1)^2 - (R+2).
\end{equation} From Proposition \ref{prop:SernesiY} we know that $[Y] \in \widehat{\mathcal{I}}$ general has $h^1(Y, N_{Y/\Pp^{R-1}}) =0$; moreover, since 
$R = e - \gamma +1$, it is a straightforward computation to notice that \eqref{eq:dimZbis} equals the expression in \eqref{eq:dimHW} with the choice 
${\mathcal W} = \widehat{\mathcal{I}}$, $r = R-1$, $d = e$ and $g = \gamma$, namely 
\begin{equation}\label{eq:dimZ}
\dim\;{\mathcal H} (\widehat{\mathcal{I}}) = R (e+1) - (R-4) (\gamma-1) = \lambda_{e,\gamma,R-1} +R. 
\end{equation}

\begin{remark}\label{rem:parametric} The previous parametric description of ${\mathcal H} (\widehat{\mathcal{I}})$ can be formalized by taking into account the schematic construction of 
${\mathcal H} (\widehat{\mathcal{I}})$ which deals with universal Picard varieties over $\mathcal M_{\gamma}$.  To do so, 
we follow procedure as in \cite[\S\,2]{Ciro}.  Let ${\mathcal M}^0_{\gamma}$ be the Zariski open subset of the moduli
space ${\mathcal M}_{\gamma}$, whose points correspond to isomorphism
classes of curves of genus $g$ without non-trivial automorphisms. By definition, ${\mathcal M}^0_{\gamma}$ is a fine moduli space, i.e.\ it has 
a universal family $p : {\mathcal C} \to {\mathcal M}^0_{\gamma}$,
where ${\mathcal C}$ and ${\mathcal M}^0_{\gamma}$ are smooth schemes and
$p$ is a smooth morphism. ${\mathcal C} $ can be identified with
the Zariski open subset ${\mathcal M}^0_{\gamma,1}$ of the moduli space
${\mathcal M}_{\gamma,1}$ of smooth, $1$--pointed, genus--$\gamma$ curves, whose
points correspond to isomorphism classes of pairs $[(C,x)]$, with $x
\in C$ a point and $C$ a smooth curve of genus $\gamma$ without non-trivial
automorphisms. On  ${\mathcal M}^0_{\gamma,1}$ there is again a
universal family $p_1 : {\mathcal C}_1 \to {\mathcal M}^0_{\gamma,1}$,
where ${\mathcal C}_1 = {\mathcal C} \times_{{\mathcal M}^0_{\gamma}}
{\mathcal C}$. The family $p_1$ has a natural regular global section $\delta$ whose image is the diagonal. 
By means of $\delta$, for any integer $k$, we have the {\em universal family of
Picard varieties of order $k$ over} ${\mathcal M}^0_{\gamma,1}$, i.e.\ $$p_1^{(k)} : {\mathcal
Pic}^{(k)} \to {\mathcal M}^0_{\gamma,1}$$(cf.\;\cite[\S\,2]{Ciro}) and, setting $\mathcal Z_k := {\mathcal C}_1
\times_{{\mathcal M}^0_{\gamma,1}}{\mathcal Pic}^{(k)},$ we have a Poincar\`e line-bundle ${\mathcal L}_k$ on $ \mathcal Z_k$ (cf.\;a relative
version of \cite[p. 166-167]{ACGH}). For any closed point $[(C,x)] \in {\mathcal M}^0_{\gamma,1}$, its fibre via $p_1^{(k)}$ is
isomorphic to ${\rm Pic}^{(k)}(C)$.

Take $k = e \geqslant 2\gamma -1$ and let $\pi_2 : \mathcal Z_e \to {\mathcal Pic}^{(e)}$ be the projection onto the second factor. 
For a general point \linebreak $u := [(C,x),{\mathcal O}_C(E)] \in {\mathcal Pic}^{(e)}$, the restriction of ${\mathcal L}_e$ to $\pi_2^{-1}(u)$ is isomorphic to 
${\mathcal O}_C(E) \in {\rm Pic}^e(C)$ for $[(C,x)] \in {\mathcal M}^0_{g,1}$ general;  one has  
$\mathcal E_e := {\mathcal O}_{\mathcal Z} \oplus {\mathcal L}_e$ as a rank--two vector bundle on $\mathcal Z_e$. 

The fibre of $\mathcal E_e$ over $u = [(C,x),{\mathcal O}_C(E)] \in {\mathcal Pic}^{(e)}$ is the rank--two vector bundle 
$\mathfrak{E}_u = \mathfrak{F}_u (E) := {\mathcal O}_C \oplus {\mathcal O}_C(E)$ on $C$, where ${\mathfrak F}_u = {\mathcal O}_C (-E) \oplus {\mathcal O}_C$ as in \S\;\ref{cones} and where 
$[(C,x)] \in {\mathcal M}^0_{\gamma,1}$ is general. Moreover, the 
sheaf $(\pi_2)_*(\mathcal E_e)$ is free of rank $R+1 = e - \gamma +2$ on a suitable dense, open subset $\mathcal U$ of $ {\mathcal Pic}^{(e)} $; therefore, on $ \mathcal U$, we have functions 
$s_0, \ldots, s_R$ such that, for each point $ u \in \mathcal U$, $s_0, \ldots, s_R$ computed at $u = [(C,x),{\mathcal O}_C(E)]$ span the space of sections of the corresponding vector 
bundle $\mathfrak{E}_u= \mathfrak{F}_u (E)$.

There is a natural morphism $$\Psi_{e}: {\mathcal Pic}^{(e)} \times {\rm PGL}(R+1, \mathbb{C}) \to {\rm Hilb}(e,\gamma,R),$$where ${\rm Hilb}(e,\gamma,R)$ denotes the Hilbert scheme of surfaces in ${\mathbb P}^R$ of degree $e$ and sectional genus $\gamma$: given a pair $(u,\omega)$, embed $S_u := \Pp(\mathfrak{E}_u)$ to $\Pp^R$ via the sections $s_0, \ldots, s_R$
computed at $u$, compose with the projectivity $\omega$ and take the image. Since ${\mathcal Pic}^{(e)} \times {\rm PGL}(R+1, \mathbb{C}) $ is irreducible, by Proposition \ref{prop:Flam7}, 
${\mathcal H} (\widehat{\mathcal{I}})$ is the closure of the image of the above map to the Hilbert scheme. By construction,
${\mathcal H} (\widehat{\mathcal{I}})$ dominates ${\mathcal M}_{\gamma}$ and its general point
represents a cone $F \subset \Pp^{R}$ as in Proposition \ref{prop:Flam7}. From the previous construction, for $[F] \in {\mathcal H} (\widehat{\mathcal{I}})$ general, 
one has $\dim \; \Psi^{-1}_{e} ([F]) = \dim G_F + 1$. From  Lemma \ref{lem:step2}, one has $\dim\,G_F = R+1$ so $\dim\;{\mathcal H} (\widehat{\mathcal{I}}) = 4\gamma - 3 + (R+1)^2 - (R+2)$ 
as in \eqref{eq:dimZbis}. 
\end{remark}

\vskip 15pt


\subsection{Curves on cones and ramified coverings}\label{curvesramif} In this section, we construct suitable ramified $m$--covers of \linebreak $Y \subset \Pp^{R-1}$, for $[Y] \in \widehat{\mathcal{I}}$ general in the distinguished component $\widehat{\mathcal{I}}$, with the use of cones $F$ parametrized by ${\mathcal H} (\widehat{\mathcal{I}})$. Our approach extends the strategy used in \cite{CIK20}, which deals with double covers.

Using notation and assumptions as in Proposition \ref{prop:Flam7}, for any integer  $m\geqslant 1$,
let $C_m \in |\mathcal{O}_S(m C_1)|$ be a general member of the linear system on $S$ and let $X_m := \Psi(C_m) \subset F$
denote its image.

\begin{proposition}\label{prop:Flam7e} For any integer $m \geqslant 1$, one has:

\noindent
(i) $X_m$ is a smooth, irreducible curve of degree $\deg\, X_m = me$, which is non--degenerate and linearly normal in $\Pp^R$.

\vskip 3pt

\noindent
(ii) $X_m$ is obtained by the intersection of the cone $F$ with a hypersurface of degree $m$ in $\Pp^R$.

\vskip 3pt

\noindent
(iii) The projection $\pi_v$ from the vertex $v \in F$ gives rise to a morphim  $\varphi_m : X_m \to Y$, which is a degree--$m$ covering map
induced on $X_m$ by the ruling of the cone $F$.

\vskip 3pt

\noindent
(iv) The geometric genus of $X_m$ is
\begin{equation}\label{genusg}
g_m : = g(X_m) = m(\gamma-1)+\frac{m(m-1)}{2}e +1.
\end{equation}

\vskip 3pt

\noindent
(v) For any $j \geqslant m$, the line bundle ${\mathcal O}_{X_m} (j)$ is non--special and such that 
\begin{equation}\label{eq:h0m}
h^0(X_m, {\mathcal O}_{X_m} (j)) = jme - m(\gamma-1) - \frac{m(m-1)}{2}e, \;\; \forall\;\; j \geqslant m.
\end{equation}
\end{proposition}

\begin{proof} For $m=1$, $X_1 =Y$ as in Proposition \ref{prop:Flam7} and there is nothing else to prove. Therefore, from now on we will focus on $m \geqslant 2$.

\vskip 3pt

\noindent
(i) Since $C_m$ is a smooth, irreducible curve on $S$ and since  $C_m \cdot C_0 = m C_1 \cdot C_0 = 0$,  then   $C_m$ is isomorphically embedded via $\Psi$ onto its image $X_m\subset F \subset \Pp^R$, which does not pass through the vertex $v \in F$.   Moreover,
$\deg\,X_m = C_m \cdot C_1 = m C_1 \cdot C_1 = m C_1^2 = me$.

Tensoring  the exact sequence defining $C_m$ on $S$ by $\mathcal{O}_S (C_1)$, we get
$$0 \to {\mathcal O}_S ((1-m) C_1) \to {\mathcal O}_S(C_1) \stackrel{r_{C_1}}{\longrightarrow} {\mathcal O}_{C_m}(C_1) \cong {\mathcal O}_{X_m}(1) \to 0.$$
Since $m \geqslant 2$, then $h^0({\mathcal O}_S ((1-m) C_1)) = 0$. Moreover, by Serre duality,
$$h^1(S, {\mathcal O}_S ((1-m) C_1)) = h^1(S, \omega_S \otimes {\mathcal O}_S ((m-1) C_1)).$$
From the facts that
$\Psi$ is birational, $C_1^2 = e >0$ and $(m-1) >0$, it follows that
${\mathcal O}_S ((m-1) C_1)$ is big and nef, so $h^1(S, \omega_S \otimes {\mathcal O}_S ((m-1) C_1)) = 0$, by  Kawamata--Viehweg vanishing theorem.
Thus, $$H^0(X_m, {\mathcal O}_{X_m}(1)) \cong H^0(C_m, {\mathcal O}_{C_m}(C_1)) \cong H^0(S, {\mathcal O}_S(C_1))$$which implies that $X_m$ is non--degenerate and linearly normal,
as it follows from Proposition \ref{prop:Flam7}--(iv).

\vskip 3pt

\noindent
(ii) Since $C_m \sim mC_1$ on $S$ and since $C_1$ induces the hyperplane section of $F$, it follows that $X_m \in | {\mathcal O}_{F}(m)|$.

\vskip 3pt

\noindent
(iii) Taking into account diagram \eqref{eq:diagPaola}, the projection from the vertex $v$ induces the morphim  $\varphi_m$. Since
$C_m \cdot f = m C_1 \cdot f = m$ and since any fiber $f$ is embedded by $\Psi$ as a line of $F$, $\varphi_m$ is induced by the ruling of the cone. 
As $Y$ is smooth, irreducible and all the fibers of $\varphi_m$ have constant length $m$, then $\varphi_m$ is a finite, flat morphism from $X_m$ to $Y$ (cf.\;e.g.\;\cite{S06}). 
Therefore, $\varphi_m$ is a covering map of degree $m$ as in \S\,\ref{Coverings}.

\vskip 3pt

\noindent
(iv) The genus of $X_m$ equals the genus of $C_m$. Therefore, to compute $g_m$ we can apply adjunction formula on $S$ and the Riemann-Hurwitz formula as in \eqref{eq:RiemannH}
to the map $\varphi_m: C_m \to C_1$ induced by the fibers of the ruling of $S$ (to ease notation, we use the same symbol as for the
map $\varphi_m : X_m \to Y$ induced by the projection from the vertex $v$ of the cone $F$).

If $R_{\varphi_m}$ denotes the {\em ramification divisor} of $\varphi_m$ on $C_m$, by Riemann--Hurwitz \eqref{eq:RiemannH} one has
$ {\mathcal O}_{C_m}(R_{\varphi_m}) \cong {\mathcal O}_{C_m}(K_{C_m} - \varphi_m^* (K_{C_1}))$. By adjunction formula, for $j = 1,\,m$,  
the canonical divisor $K_{C_j}$ is induced on $C_j$ by the divisor $K_S + C_j$ on $S$, which is  
$$K_S + C_j \sim (j-2) C_0 + (j-1) Ef + K_Cf, \;\;\; j = 1, m.$$Therefore one has 
\begin{equation}\label{eq:ramifdiv}
{\mathcal O}_{C_m} (R_{\varphi_m})\cong {\mathcal O}_{C_m}((m-1) C_1) \;\;\; {\rm and} \;\;\;
\deg\, R_{\varphi_m} = (m-1) C_1 \cdot C_m = m (m-1) C_1^2 = m (m-1)e.
\end{equation}Using Riemann--Hurwitz formula \eqref{eq:RiemannH}, one gets therefore
$$2 g_m - 2 = m (2 \gamma -2) + m (m-1) e$$which gives \eqref{genusg}.

\vskip 3pt

\noindent
(v) Since $\deg\, X_m = me$, then $\deg\; {\mathcal O}_{X_m}(j) = jme$ whereas, from above, $\deg\; \omega_{X_m} = 2 g_m-2 = 2m(\gamma-1) + m (m-1) e$.
Since $e \geqslant 2\gamma -1$, it is a straightforward computation to notice that if $j \geqslant m$ then 
$$\deg\; {\mathcal O}_{X_m}(j) >  \deg\; \omega_{X_m},$$which implies the non--speciality of ${\mathcal O}_{X_m}(j)$ for any $j \geqslant m$.
The computation of $ h^0(X_m, {\mathcal O}_{X_m}(j))$ then reduces to simply apply Riemann--Roch on the curve $X_m$.
\end{proof}

We conclude the section with a general result, involving covering maps and projections, which in particular applies to smooth, irreducible curves $X_m$ on $F$ as above and which extends 
\cite[Lemma\,4,\;Cor.\,5]{CIK20} to the reducible, connected case. This will be used in the proof of our Main Theorem (cf. proof of Claim \ref{cl:induction}).

\begin{lemma} \label{lem:gradi} Let $Z \subset \Pp^R$ be a non--degenerate, connected, projective curve, which is possibly reducible and which has at most nodes as possible singularities. Let
$D =\mathrm{Sing}(Z)$ denote its scheme of nodes, whose cardinality we denote by $\delta$ (in particular $\delta =0$ and $D=\emptyset$ if e.g. $Z$ smooth and irreducible).
Let $H$ be a hyperplane in $\Pp^R$ and  $ v \in \mathbb{P}^R\setminus (H\cup Z)$ be a point.  Assume that the projection $\pi_v : Z \to  H \cong \mathbb{P}^{R-1}$
from the point $v$ is such that $Y:=\pi_v(Z) \subset H$ is smooth and irreducible.  Let $R_{\pi_v}$ be the {\em ramification divisor} of $\pi_v$.

Then  $R_{\pi_v}$ is a Cartier divisor on $Z$ and  the following exact sequence holds
\begin{equation}\label{eq:Flam**}
0 \to \mathcal{L}_Z \to N_{Z/\mathbb{P}^R} \to \pi^*_v(N_{Y/\mathbb{P}^{R-1}}) \to 0,
  \end{equation}
where $N_{Z/\mathbb{P}^R}$ denotes the normal sheaf of $Z$, which is locally free on $Z$, and  $\mathcal{L}_Z$ is a line bundle on $Z$ such that
\[
\deg\,\mathcal{L}_Z =\deg\, Z +\deg \; R_{\pi_v} + \delta.
\]

If, in particular, $Z = X_m$ as in Proposition \ref{prop:Flam7e}, for some $m \geqslant 2$, and $v$ is the vertex of the cone $F=F_Y$ then $\pi_v(X_m) = Y$, where
$Y$ a hyperplane section of $F$ not passing through $v$ as in Proposition \ref{prop:Flam7}, and \eqref{eq:Flam**} reads
\begin{equation}\label{eq:Flam***}
0 \to {\mathcal O}_{X_m} (R_{\pi_v}) \otimes {\mathcal O}_{X_m}(1) \to N_{X_m/\mathbb{P}^R} \to \pi^*_v(N_{Y/\mathbb{P}^{R-1}}) \to 0.
  \end{equation}
\end{lemma}
\begin{proof} If $\mathcal{I}_{Z/\mathbb{P}^{R}}$ denotes the ideal sheaf of $Z$ in $\Pp^R$ then, since $Z$ is at most nodal, 
$N_{Z/\mathbb{P}^{R}} := {\mathcal Hom}({\mathcal I}_{Z/\Pp^R}, \mathcal{O}_Z)$ and $T_{\mathbb{P}^{R}}|_Z := {\mathcal Hom}(\Omega^1_{\mathbb{P}^{R}}, \mathcal{O}_Z) $ are
both locally--free of rank $R-1$ and $R$, respectively (cf.\;\cite[page\,30]{S84}).

If we take into account the projection $\pi_v: Z \to Y \subset H$, one has $\pi^*_v(\mathcal{O}_Y) \cong \mathcal{O}_Z$ and 
$\pi^*_v(\mathcal{O}_Y(1)) \cong \mathcal{O}_Z(1)$; thus, considering  the Euler sequences of $Z$ and $Y$
\begin{eqnarray}
  &0& \to \mathcal{O}_Z\to \mathcal{O}_Z(1)^{\oplus (R+1)}\to T_{\mathbb{P}^{R}}|_Z\to 0 \nonumber
  \\
  &0& \to \mathcal{O}_Y\to \mathcal{O}_Y(1)^{\oplus R} \to T_{\mathbb{P}^{R-1}}|_Y\to 0 \nonumber
\end{eqnarray} and pulling--back to $Z$ via $\pi_v$ the second Euler sequence, one deduces the  following exact diagram
\\
\[
\begin{tikzcd}
 &                  & 0                 \arrow{d}                 & 0                 \arrow{d}          &
  \\
       & 0 \arrow{d} & \mathcal{O}_Z(1) \arrow{d}  & \mathcal{K}er\,\alpha \arrow{d} & 
  \\
0 \arrow{r} & \mathcal{O}_Z \arrow{r}\arrow[equal]{d}   & \mathcal{O}_Z(1)^{\oplus (R+1)} \arrow{r}\arrow{d}  & T_{\mathbb{P}^{R}}|_Z   \arrow{r}\arrow[d,"\alpha"] & 0
   \\
 0 \arrow{r} & \mathcal{O}_Z \arrow{r}\arrow{d} & \mathcal{O}_Z(1)^{\oplus R} \arrow{r}\arrow{d} & \pi_v^*(T_{\mathbb{P}^{R-1}}|_Y)  \arrow{r}\arrow{d}  & 0
   \\
  & 0                   & 0                                    & 0 
\end{tikzcd}
\\
\]where the map $\alpha$ is surjective by the Snake Lemma. The exactness of the diagram implies that
$\mathcal{K}er\,\alpha \cong \mathcal{O}_Z(1)$. Hence, from the right--most exact column of the diagram, we get
\begin{eqnarray}\label{AA}
 0 \to \mathcal{O}_Z(1)\to  T_{\mathbb{P}^{R}}|_Z \to\pi_v^*(T_{\mathbb{P}^{R-1}}|_Y)\to 0.
\end{eqnarray}

If we set $\Omega^1_Z$ the {\em cotangent sheaf} (or the {\em sheaf of K\"ahler differentials}) on $Z$,  then its dual
$\Theta_{Z} := {\mathcal Hom}(\Omega^1_{Z}, \mathcal{O}_Z) $ is not locally--free, but it is torsion free (cf. \cite{S84}) and it is called the {\em sheaf of derivations} of $\mathcal O_Z$ (when
$Z$ is smooth and irreducible, $\Omega_Z^1$ coincides with the {\em canonical bundle} whereas $\Theta_Z$ with the {\em tangent bundle}).
At last, $T^1_Z :={\mathcal Ext}^1(\Omega^1_{Z}, \mathcal{O}_Z)$ is called
the {\em first cotangent sheaf} of $Z$, which is a torsion sheaf supported on $Sing(Z)$. Since by assumption $Z$ is at most nodal, it is  either
$T^1_Z = 0$ (i.e. the zero--sheaf) when $Z$ is smooth, or $T^1_Z \cong {\mathcal O}_D$, where $D$ the set of nodes of $Z$, otherwise.
By \cite[page\,30, (1.2)]{S84}, one has the exact sequence
\begin{equation}\label{BB}
0\to \Theta_{Z} \to T_{\mathbb{P}^{R}}|_Z \to N_{Z/\mathbb{P}^{R}}\to T^1_Z \to 0.
\end{equation}
Putting together \eqref{AA} and \eqref{BB} and taking into account the map $\pi_v : Z \to Y$, one gets the following exact diagram:
\[
\begin{tikzcd}
 &   0                 \arrow{d}                 & 0                 \arrow{d}                 & 0                 \arrow{d}          & &
  \\
   & \mathcal{K}er\,\pi_{v*} \arrow{d} & \mathcal{O}_Z(1) \arrow{d}  & \mathcal{K}er\,\beta\arrow{d} &  &
  \\
  0\arrow{r} &\Theta_{Z} \arrow{r} \arrow[d,"\pi_{v*}"]  &T_{\mathbb{P}^{R}}|_Z \arrow{r}\arrow{d} &  N_{Z/\mathbb{P}^{
  R}} \arrow{r} \arrow[d,"\beta"] & T^1_Z \arrow{r} & 0
  \\
 0 \arrow{r} & \pi^*_v(T_{Y}) \arrow{r}\arrow{d} & \pi^*_v(T_{\mathbb{P}^{R-1}}|_Y) \arrow{r}\arrow{d} & \pi^*_v(N_{Y|\mathbb{P}^{R-1}})  \arrow{r}\arrow{d}  & 0   &
   \\
  & \mathcal{C}oker\, \pi_{v*} \arrow{d} & 0  & 0  &   &
  \\
 &0 &   && &
\end{tikzcd}
\]where $\beta$ is defined by the diagram. From \cite{S84}, the sequence \eqref{BB} splits in two exact sequences
\begin{eqnarray}\label{CC}
&&
0\to \Theta_{Z}\to T_{\mathbb{P}^{R}}|_Z\to N'_{Z}\to 0 \nonumber
\\
&&
0\to N'_{Z}\to   N_{Z/\mathbb{P}^{R}}\to T^1_Z \to 0, \nonumber
\end{eqnarray}
where $N'_{Z}$ is the {\em equi--singular sheaf}. Hence, the previous exact diagram gives rise to  the following:
\begin{equation}\label{diaDelta}
\begin{tikzcd}
 &   0                 \arrow{d}                 & 0                 \arrow{d}                 & 0                 \arrow{d}          &
  \\
       & \mathcal{K}er\,\pi_{v*} \arrow{d} & \mathcal{O}_Z(1) \arrow{d}  & \mathcal{K}er\,\beta' \arrow{d} & 
  \\
0 \arrow{r} &  \Theta_{Z} \arrow{r}\arrow[d,"\pi_{v*}"]   & T_{\mathbb{P}^{R}}|_Z\arrow{r}\arrow{d}  & N'_{Z} \arrow{r} \arrow[d, "\beta'"] & 0
   \\
 0 \arrow{r} & \pi^*_v(T_{Y}) \arrow{r}\arrow{d} & \pi^*_v(T_{\mathbb{P}^{R-1}}|_Y) \arrow{r}\arrow{d} & \pi^*_v(N_{Y/\mathbb{P}^{R-1}})  \arrow{r}\arrow{d}  & 0
   \\
  & \mathcal{C}oker\, \pi_{v*} \arrow{d} & 0  & 0 &
  \\
 &0 &   &&
\end{tikzcd}
\end{equation}where $\beta'$ is induced by $\beta$ from the previous diagram. By the Snake Lemma, one has therefore
\begin{eqnarray}\label{DD}
&&
0\to \mathcal{K}er\, \pi_{v*} \to \mathcal{O}_Z(1)  \to \mathcal{K}er\, \beta'\to \mathcal{C}oker\, \pi_{v*} \to 0.
\end{eqnarray}
Since $\pi_{v}^*(T_Y)$ is a line bundle on $Z$ and $\Theta_{Z}$ has generically rank $1$ on $Z$, if it were $\mathcal{K}er\, \pi_{v*} \neq 0$ then it would be a torsion sheaf, which is a contradiction
by \eqref{DD} and the fact that $\mathcal{O}_Z(1)$ is a line bundle on $Z$. Therefore \eqref{DD} gives
\begin{eqnarray}\label{DDDD}
&&
0\to \mathcal{ O}_Z(1)  \to \mathcal{K}er\, \beta'\to \mathcal{C}oker\, \pi_{v*} \to 0.
\end{eqnarray}
Since by the right--most column of diagram \eqref{diaDelta} the sheaf $\mathcal{K}er\, \beta'$ has generically rank $1$ and, by
the left--most column of diagram \eqref{diaDelta}, $\mathcal{C}oker\, \pi_{v*}$ is a torsion sheaf, from
\eqref{DDDD} it follows that $\mathcal{K}er\, \beta'$ is a line bundle whereas
$\mathcal{C}oker\, \pi_{v*} \cong \mathcal{O}_{R_{\pi_v}}$, and $R_{\pi_v}$ is the (effective, Cartier) ramification
divisor of the projection $\pi_v$. Thus, $\mathcal{K}er\, \beta'$ is a line bundle on $Z$ such that, from \eqref{DDDD},
is isomorphic to $\mathcal{O}_Z( R_{\pi_v})\otimes \mathcal{O}_Z(1)$. Therefore, the sequence \eqref{DDDD} reads as
\begin{eqnarray}\label{DDDDDD}
&&
0\to \mathcal{ O}_Z(1)  \to\mathcal{O}_Z( R_{\pi_v})\otimes \mathcal{O}_Z(1)\to\mathcal{O}_{R_{\pi_v}}\to 0.
\end{eqnarray}

From diagram \eqref{diaDelta} we deduce
\[
\begin{tikzcd}
 &   0                 \arrow{d}                 & 0                 \arrow{d}                 & 0                 \arrow{d}          &
  \\
 0  \arrow{r}   & \mathcal{O}_Z( R_{\pi_v})\otimes\mathcal{O}_Z(1) \arrow{r}\arrow{d} & \mathcal{K}er\, \beta \arrow{r}\arrow{d}  & T^1_Z\cong \mathcal{O}_D \arrow{r}\arrow[equal]{d} & 0
  \\
0 \arrow{r} &  N'_{Z} \arrow{r}\arrow[d,"\beta'"]   & N_{Z/\mathbb{P}^{R}}\arrow{r}\arrow[d,"\beta"]  & T^1_Z \arrow{r} \arrow{d} & 0
   \\
 0 \arrow{r} & \pi^*_v(N_{Y/\mathbb{P}^{R-1}}) \arrow[equal]{r}\arrow{d} & \pi^*_v(N_{Y/\mathbb{P}^{R-1}}) \arrow{r}\arrow{d} & 0 \arrow{r}\arrow{d}  & 0
   \\
  & 0 & 0  & 0 &
  \end{tikzcd}
\]By the Snake Lemma again,  $ \mathcal{K}er\, \beta =: \mathcal{L}_Z$ is a line bundle too, for which
\begin{eqnarray}\label{EE}
&&
0\to \mathcal{O}_Z( R_{\pi_v})\otimes \mathcal{O}_Z(1) \to \mathcal{L}_Z \to T_Z^1\cong\mathcal{O}_D \to 0.
\end{eqnarray} holds. In particular, by \eqref{EE} one has
$$\deg \;\mathcal{L}_Z = \deg \; \left(\mathcal{O}_Z( R_{\pi_v})\otimes \mathcal{O}_Z(1) \right) + \delta = \deg\, Z + \deg \, R_{\pi_v} + \delta$$as stated.
Moreover, from the middle--column of the above diagram, one also has 
\[
0\to \mathcal{L}_Z \to N_{Z/\mathbb{P}^{R}} \to  \pi^*_v(N_{Y/\mathbb{P}^{R-1}}) \to 0,
\]  and this concludes the first part of the statement.

At last, if $Z = X_m \subset F$ as in Proposition \ref{prop:Flam7e} and if $v$ is the vertex of the cone $F$, the projection $\pi_v$ induces a  $m$-sheeted ramified cover of the base curve $Y$
which is a hyperplane section of $F$ not passing through the vertex $v$; in this case $\delta =0$, $D= \emptyset$ and
$\mathcal{L}_{X_m} \cong \mathcal{O}_{X_m}( R_{\pi_v})\otimes \mathcal{O}_{X_m}(1)$ so \eqref{eq:Flam**} becomes
\eqref{eq:Flam***}, as stated.
\end{proof}

\vskip 15pt


\section{Superabundant components of Hilbert schemes}\label{superabundant} This section is entirely devoted to the construction of {\em superabundant} components of Hilbert schemes and to
the proof of our Main Theorem.  To do so, we will need to deal with the surjectivity of the Gaussian--Wahl map
$\Phi_{Y,{\mathcal O}_Y(1)}$ for $Y \subset \Pp^{R-1}$ as in \S\,\ref{curves} (cf.\;Claim\;\ref{cl:induction} below).

\begin{remark}\label{impo} {\rm Recall that ${\mathcal O}_Y(1) \cong {\mathcal O}_C(E)$ is of degree $e \geqslant 2\gamma -1$. Taking into account numerical assumptions as 
in Proposition \ref{prop:CLM2.1}, for $6 \leqslant \gamma \leqslant 8$, the condition $2\gamma - 1 \geqslant \gamma + 12$ cannot hold since it would give 
$\gamma \geqslant 13$, contradicting that $6 \leqslant \gamma \leqslant 8$; similarly, condition $2\gamma-1 \geqslant \gamma + 9$ does not hold for 
$\gamma=9$. On the contrary, $\gamma \geqslant 10$ ensures that $\deg {\mathcal O}_C(E) = e \geqslant 2\gamma -1 \geqslant \gamma + 9$ holds true so we can apply 
Proposition \ref{prop:CLM2.1} to the pair $(C, {\mathcal O}_C(E))$ giving rise to $Y$ to prove the next result. 
}
\end{remark}

\begin{proposition}\label{prop:2.1(2.8)Y}  Let $\gamma \geqslant 10$, $e \geqslant 2 \gamma - 1$ and $R = e - \gamma +1$ be integers. 
Let $\widehat{{\mathcal I}_{e, \gamma, R-1}}$ be the distinguished component of $\mathcal{I}_{e, \gamma, R-1}$ and let $[Y] \in \widehat{{\mathcal I}_{e, \gamma, R-1}}$ 
be general. Then: 

\noindent
$(a)$ the Gaussian-Wahl map $\Phi_{\omega_Y,{\mathcal O}_Y(1)}$ is surjective, 

\noindent
$(b)$ ${\mathcal H} (\widehat{{\mathcal I}_{e, \gamma, R-1}})$ is generically smooth of dimension
\begin{equation}\label{eq:dimHI}
\dim\, {\mathcal H} (\widehat{{\mathcal I}_{e, \gamma, R-1}}) = \lambda_{e, \gamma, R-1} + R = R (e+1) - (R-4) (\gamma-1), 
\end{equation}and 

\noindent
$(c)$ the cone $F$, corresponding to $[F] \in {\mathcal H} (\widehat{{\mathcal I}_{e, \gamma, R-1}})$ general, is unobstructed in $\Pp^R$.
\end{proposition}
\begin{proof} As observed in Remark \ref{impo}, $\gamma \geqslant 10$ and $e \geqslant 2\gamma -1$ imply that 
numerical assumptions of Proposition \ref{prop:CLM2.1} certainly hold. Moreover, since the pair  $(C, {\mathcal O}_C(E))$, giving rise to $Y$, is such that 
$C$ is with general moduli and ${\mathcal O}_C(E) \in {\rm Pic}^e(C)$ is general we are in position to apply Propositions  \ref{prop:CLM2.1} and \ref{prop:CLM2.18}, from which 
$(a)$, $(b)$ and $(c)$ directly follow. 
\end{proof}

\noindent
Notice that \eqref{eq:dimHI} coincides with the expression \eqref{eq:dimZ}, which has been independently 
found via the parametric description of ${\mathcal H} (\widehat{{\mathcal I}_{e, \gamma, R-1}})$ in \S\;\ref{cones}.

From Remark \ref{impo} and Proposition \ref{prop:2.1(2.8)Y}, we therefore fix from now on the following numerical assumptions:
\begin{equation}\label{eq:ass1}
\gamma \geqslant 10, \;\;\; e \geqslant 2 \gamma -1, \;\;\; R = e - \gamma + 1, \;\;\; m \geqslant 2.
\end{equation}Furthermore, to ease notation, we simply pose:
\begin{equation}\label{eq:ass2}
d:= me,\;\;\; X:= X_m, \;\;\; g:= g_m,
\end{equation} where $X_m$ and $g_m$ are as in Proposition \ref{prop:Flam7e}. 

In this set--up we have that $[X] \in {\mathcal I}_{d,g,R}$; we now show that,  
as $[F]$ varies in  ${\mathcal H} (\widehat{{\mathcal I}_{e, \gamma, R-1}})$, curves $X$ fill--up an irreducible locus in  
${\mathcal I}_{d,g,R}$ as  follows. With notation as in \S\;\ref{cones}, set first
\begin{align*}
\mathcal{U}_{e,\gamma,R}: = &\Big\{u:= \left([C], \; \mathcal{O}_C(E), \; S, \; C_1\right) \;\; | \; \; [C]\in \mathcal{M}_{\gamma} \; \text{general}, \; \mathcal{O}_C(E)\in \mathrm{Pic}^e(C) \; \text{general}, \\
\;\; &  \;\;\; \mathfrak{F} = \mathcal{O}_C\oplus  \mathcal{O}_C(-E),\; S = \Pp(\mathfrak{F}), \;  \; C_1 \in |\mathcal{O}_S(C_0 + Ef)| \; \text{general} \Big\}; 
\end{align*}by construction $\mathcal{U}_{e,\gamma,R}$ is obviously irreducible. Then, for any $m \geqslant 2$, consider
\[
\mathcal{W}_{d,g,R}: = \Big\{ (u, C_m) | \, u \in \mathcal{U}_{e,\gamma,R},\, C_m \in |\mathcal{O}_S(m(C_0 + Ef))| \;\; {\rm general} \;\; \Big\} \stackrel{\pi}{\longrightarrow} \mathcal{U}_{e,\gamma,R-1}, \;\;\; (u, C_m) \mapsto u,
\] where the natural projection map $\pi$ endows $\mathcal{W}_{d,g,R}$ with a structure of a non--empty, open dense subset of a projective--bundle over $\mathcal{U}_{e,\gamma,R}$, hence $\mathcal{W}_{d,g,R}$ is irreducible too  (recall that $d = me$ and $g = g_m$ depend on $m$).

By the very definition of $\mathcal{W}_{d,g,R}$, one has a natural {\em Hilbert morphism}
\begin{eqnarray*}
h:  \mathcal{W}_{d,g,R}  &\longrightarrow & \mathcal{I}_{d, g, R}
\\
  (u, C_m) &\mapsto & [X_m]:= [\Psi(C_m)],
\end{eqnarray*} where $\Psi$ the morphism as in Proposition \ref{prop:Flam7}, and one defines 
\begin{equation}\label{eq:Ddgr}
{\mathcal S}_{d,g,R} : = h(\mathcal{W}_{d,g,R})\subset \mathcal{I}_{d, g, R}.
\end{equation}

\begin{lemma}\label{lem:SdgR} ${\mathcal S}_{d,g,R}$ is irreducible, it has dimension
\begin{equation}\label{eq:DdgrBIS}
\dim\; {\mathcal S}_{d,g,R} = \lambda_{d,g,R} + \sigma_{d,g,R},
\end{equation}where
$$\lambda_{d,g,R} = (R+1) me  - (R-3) \left( m(\gamma-1)+\frac{m(m-1)}{2}e  \right)$$is the expected dimension of ${\mathcal I}_{d,g,R}$ as in \eqref{eq:expdimHilb},  
whereas the positive integer 
$$\sigma_{d,g,r} := (R-4) \left[(\gamma-1)(m-1) + 1 + e + \frac{m(m-3)}{2}e\right] + 4 (e+1) + e m (m-5)$$is called 
the {\em superabundance summand} of the dimension of ${\mathcal S}_{d,g,R}$. 

Furthermore, ${\mathcal S}_{d,g,R}$ is generically smooth. 
\end{lemma}

\begin{proof} By construction of ${\mathcal S}_{d,g,R}$, it is irreducible and 
$$\dim\; {\mathcal S}_{d,g,R} = \dim\; {\mathcal H} (\widehat{{\mathcal I}_{e, \gamma, R-1}}) + \dim\;|\mathcal{O}_F(m)|,$$where $[F] \in {\mathcal H}(\widehat{{\mathcal I}_{e, \gamma, R-1}})$ is general. 
Thus, from \eqref{eq:dimHI} (equivalently, from \eqref{eq:dimZ}) and from Proposition \ref{prop:Flam7} (v), the latter reads
\begin{equation}\label{eq:dimDdgr}
\begin{array}{ccl}
\dim\; {\mathcal S}_{d,g,R} & = & R(e+1) - (R-4) (\gamma-1) + \frac{m(m+1)}{2}e -m(\gamma-1) = \\
& = &  \lambda_{e,\gamma,R-1} \, + R + \frac{m(m+1)}{2}e -m(\gamma-1).
\end{array}
\end{equation}
Taking into account \eqref{eq:expdimHilb} which, in our notation, reads
$$
\lambda_{d,g,R} = (R+1) me  - (R-3) \left( m(\gamma-1)+\frac{m(m-1)}{2}e  \right),
$$ 
to prove the first part of the statement it suffices to showing that $\dim\; {\mathcal S}_{d,g,R} - \lambda_{d,g,R} = \sigma_{d,g,r}$ and that the latter integer is positive. 

To do so, observe that
$$\dim\; {\mathcal S}_{d,g,R} - \lambda_{d,g,R} = R(e+1) - (R-4) (\gamma-1) + \frac{m(m+1)}{2}e -m(\gamma-1) - \left[ (R+1) me  - (R-3) \left( m(\gamma-1)+\frac{m(m-1)}{2}e  \right)\right]= $$
$$= R\left(\gamma (m-1) - m + 2 +e\right) + Re \frac{m(m-3)}{2} - 4 (\gamma-1) (m-1) - e m (m-1) = $$
$$=(R-4) \left[(\gamma-1)(m-1) + 1 + e + \frac{m(m-3)}{2}e\right] + 4 (e+1) + e m (m-5).$$Notice that, since $R = e - \gamma +1$, $e \geqslant 2\gamma-1$ and $\gamma \geqslant 10$, then
$R-4 \geqslant 6$; moreover, since $m \geqslant 2$, the summands in square--parentheses add--up to a positive integer: the statement is clear for $m \geqslant 3$, whereas for $m=2$
one has $\left[(\gamma-1)(m-1) + 1 + e + \frac{m(m-3)}{2}e\right] = \gamma \geqslant 10$. Concerning the summand $4(e+1)$, in our assumptions it is $4 (e+1) \geqslant 8 \gamma \geqslant 80$.  The last summand $e m(m-5)$ is non--negative for $m \geqslant 5$, whereas for $m =2, 3,4$ it is, respectively, $-6e, -6e, -4e$; in all the latter three sporadic cases, the negativity  of the summand $e m(m-5)$ does not affect the positivity  of the total expression.

The previous computations  show that
$$\dim\; {\mathcal S}_{d,g,R} - \lambda_{d,g,R} = (R-4) \left[(\gamma-1)(m-1) + 1 + e + \frac{m(m-3)}{2}e\right] + 4 (e+1) + e m (m-5) = \sigma_{d,g,R}$$and the first part of the statement is proved.

Concerning the generic smoothness of ${\mathcal S}_{d,g,R}$, we have first the following: 

\begin{claim}\label{cl:calcoloparz} For $[X] \in {\mathcal S}_{d,g,R} $ general, one has
$$h^0(X, N_{X/\Pp^R}) =  \lambda_{e,\gamma,R-1} + h^0(Y, N_{Y/\mathbb{P}^{R-1}} \otimes \mathcal{T}^{\vee}_{\varphi}) + \frac{m(m+1)}{2} e - m(\gamma-1),$$where
$\mathcal{T}^{\vee}_{\varphi}$ is the Tschirnhausen bundle associated to the degree--$m$ covering map $\varphi : X \to Y$ induced by the
projection $\pi_v$ from the vertex $v$ of the cone $F$ as in diagram \eqref{eq:diagPaola}. 
\end{claim}
\begin{proof} [Proof of Claim \ref{cl:calcoloparz}] Consider the exact sequence \eqref{eq:Flam***} in Lemma \ref{lem:gradi} which, in the present notation, reads
$$0 \to {\mathcal O}_{X} (R_{\pi_v}) \otimes {\mathcal O}_{X}(1) \to N_{X/\mathbb{P}^R} \to \pi^*_v(N_{Y/\mathbb{P}^{R-1}}) \to 0.$$
From Proposition
\ref{prop:Flam7e} (iii), the degree--$m$ covering map $\varphi : X \to Y$ is induced by the projection $\pi_v$ from the vertex of the cone $F$ and, by
\eqref{eq:ramifdiv}, we have ${\mathcal O}_X (R_{\pi_v}) = {\mathcal O}_X (R_{\varphi}) \cong {\mathcal O}_X(m-1)$. Therefore, the previous exact sequence gives
$$0 \to {\mathcal O}_{X} (m) \to N_{X/\mathbb{P}^R} \to \varphi^*(N_{Y/\mathbb{P}^{R-1}}) \to 0.$$From Proposition \ref{prop:Flam7e} (v), ${\mathcal O}_{X} (m)$ is non--special
on $X$, so
$$h^0(X, N_{X/\mathbb{P}^R}) = h^0(X, \varphi^*(N_{Y/\mathbb{P}^{R-1}})) + h^0(X,{\mathcal O}_{X} (m)) .$$
Since $\varphi$ is a finite morphism, using Leray's isomorphism and projection formula, we get
$$h^0(X, \varphi^*(N_{Y/\mathbb{P}^{R-1}})) = h^0(Y, N_{Y/\mathbb{P}^{R-1}} \otimes \varphi_* \; \mathcal O_X).$$Moreover, from \S\,\ref{Coverings}, one has
$\varphi_*\mathcal{O}_X=\mathcal{O}_Y\oplus \mathcal{T}^{\vee}_{\varphi}$, where $\mathcal{T}^{\vee}_{\varphi}$ the Tschirnhausen bundle associated to $\varphi$. Thus,
$$h^0(Y, N_{Y/\mathbb{P}^{R-1}} \otimes \varphi_* \; \mathcal O_X) = h^0(Y, N_{Y/\mathbb{P}^{R-1}}) + h^0(Y, N_{Y/\mathbb{P}^{R-1}} \otimes \mathcal{T}^{\vee}_{\varphi}).$$

To sum--up, one has
\begin{equation}\label{eq:h0normale}
h^0(X, N_{X/\mathbb{P}^R}) = h^0(Y, N_{Y/\mathbb{P}^{R-1}}) + h^0(Y, N_{Y/\mathbb{P}^{R-1}} \otimes \mathcal{T}^{\vee}_{\varphi}) + h^0(X,{\mathcal O}_{X} (m)).
\end{equation}
By \eqref{eq:h0m} with $j=m$, one has
$$
h^0(X,{\mathcal O}_{X} (m)) = \frac{m(m+1)}{2} e - m(\gamma-1).
$$
From \eqref{eq:expdimHilb} and Corollary \ref{cor:Sernesi}, it follows that
$$h^0(Y, N_{Y/\mathbb{P}^{R-1}}) = \lambda_{e, \gamma, R-1},$$since $Y$ corresponds to a general point in the distinguished component $\widehat{\mathcal I_{e,\gamma,R-1}}$.
\end{proof}

To conclude that ${\mathcal S}_{d,g,R}$ is generically smooth, we are left with the following: 

\begin{claim}\label{cl:induction} For any $m \geqslant 2$, one has
\begin{equation}\label{eq:induction}
h^0(Y, N_{Y/\mathbb{P}^{R-1}} \otimes \mathcal{T}^{\vee}_{\varphi}) = R.
\end{equation}
\end{claim}

\begin{proof}[Proof of Claim \ref{cl:induction}] To prove the statement, we will use an inductive approach.

Assume first $m=2$, so $X= X_2$ and $\varphi := \varphi_2: X \to Y$ is the double cover of the curve $Y$, as in Proposition \ref{prop:Flam7e} (iii). In this case, the Tschirnhausen bundle $\mathcal{T}^{\vee}_{\varphi}$ is a line bundle on $Y$ which, from \eqref{eq:branch} and \eqref{eq:ramifdiv}, equals ${\mathcal O}_Y(-E) \cong {\mathcal O}_Y(-1)$.
Since $\gamma \geqslant 10$ and $e \geqslant 2\gamma -1$, assumptions of Propositions \ref{prop:CLM2.1} are satisfied. Therefore, from Proposition \ref{prop:CLM2.18} (ii), we have
$h^0( Y, N_{Y/\mathbb{P}^{R-1}} \otimes \mathcal O_Y(-1)) = R$ and \eqref{eq:induction} holds true in this case.

Take now $m \geqslant 3$ and assume that \eqref{eq:induction} holds for a degree--$(m-1)$ covering map $\varphi:= \varphi_{m-1} : X := X_{m-1} \to Y$, where $X_{m-1} \in |{\mathcal O}_F(m-1)|$ general 
as in Proposition \ref{prop:Flam7e}. To ease notation, the associated Tschirnhausen bundle $\mathcal{T}^{\vee}_{\varphi}$ will be simply denoted by $\mathcal{T}^{\vee}_{m-1}$. 

Let $Y' \in |\mathcal O_F(1)|$ be general and consider the projective, connected, non--degenerate reducible curve $Z := X \cup Y' \subset F$ which, as a Cartier divisor on $F$, 
is such that $Z \in |\mathcal O_F(m)|$. The singular locus of $Z$ is $D := X \cap Y' $ and consists of $\delta$ nodes, where $\delta:= (m-1) H^2 = (m-1) e$, $H$ denoting the 
hyperplane section of $F$.  As in \S\;\ref{Coverings}, the curve $Z$ is endowed with a natural degree--$m$ covering map $\psi: Z \to Y$, whose Tschirnhausen bundle $\mathcal{T}^{\vee}_{\psi}$ on $Y$ will be simply denoted
by $\mathcal{T}^{\vee}_{m}$.

From Proposition \ref{prop:split}, passing to duals, we get the exact sequence
$$0 \to \mathcal{O}_Y(-D)\to \mathcal{T}_{m}^{\vee} \to \mathcal{T}_{m-1}^{\vee} \to 0$$of vector bundles on $Y$. Tensoring this exact sequence with $N_{Y/\mathbb{P}^{R-1}}$ gives
\begin{equation}\label{eq:bastaaa}
 0 \to N_{Y/\mathbb{P}^{R-1}}\otimes\mathcal{O}_Y(-D)\to N_{Y/\mathbb{P}^{R-1}}\otimes\mathcal{T}_{m}^{\vee} \to N_{Y/\mathbb{P}^{R-1}}\otimes \mathcal{T}_{m-1}^{\vee} \to   0.
 \end{equation}By induction, since $m-1 \geqslant 2$, one has $h^0(Y, N_{Y/\mathbb{P}^{R-1}}\otimes \mathcal{T}_{m-1}^{\vee} ) = R$.
Moreover, since $D$ is cut--out on the irreducible component $Y'$ by a hypersurface of degree $m-1$ in $\Pp^R$ and since $Y' \cong Y$, then 
${\mathcal O}_Y(D) \cong \mathcal O_Y(m-1)$ and one has
$$h^0(Y, N_{Y/\mathbb{P}^{R-1}}\otimes \mathcal{O}_Y(-D)) =  h^0(Y,N_{Y/\mathbb{P}^{R-1}}\otimes \mathcal{O}_Y(-(m-1)) ) =0,$$as it follows from Proposition \ref{prop:CLM2.18} (iii) and from the fact that $m-1 \geqslant 2$.

By \eqref{eq:bastaaa}, we deduce that $H^0(Y, N_{Y/\mathbb{P}^{R-1}}\otimes\mathcal{T}_{m}^{\vee})$ injects into $H^0(Y, N_{Y/\mathbb{P}^{R-1}}\otimes\mathcal{T}_{m-1}^{\vee})$, so in particular
\begin{equation}\label{eq:ineq}
h^0(Y, N_{Y/\mathbb{P}^{R-1}}\otimes\mathcal{T}_{m}^{\vee}) \leqslant R.
\end{equation} On the other hand since $[Z] \in \overline{{\mathcal S}}_{d,g,R}$, where $\overline{{\mathcal S}}_{d,g,R}$ denotes the closure in $\mathcal I_{d,g,R}$ of
${\mathcal S}_{d,g,R}$, by \eqref{eq:dimhilb} one must have
\begin{equation}\label{eq:chestrazio}
h^0(Z, N_{Z/\mathbb{P}^{R}}) = \dim\; T_{[Z]} (\mathcal I_{d,g,R}) \geqslant \dim\; {\mathcal S}_{d,g,R} =  \lambda_{e,\gamma,R-1} \, + R + \frac{m(m+1)}{2}e -m(\gamma-1),
\end{equation}as it follows from \eqref{eq:dimDdgr}. Since $Z$ satisfies assumptions as in Lemma \ref{lem:gradi}, we can consider the exact sequence \eqref{eq:Flam**}. The line bundle $\mathcal{L}_Z$ therein has
degree $$\deg \; \mathcal{L}_Z = \deg\; Z + \deg\; R_{\psi} + \delta = m e + \deg\; R_{\psi}  + (m-1) e.$$By definition of
$\psi : Z \to Y$, the ramification of this map is supported on the irreducible component $X = X_{m-1}$ of $Z$, namely
$R_{\psi} = R_{\varphi_{m-1}}$ where $\varphi_{m-1} : X \to Y$. From \eqref{eq:ramifdiv} we therefore have
$\deg \; R_{\varphi_{m-1}} = (m-1)^2 e$, so
$$\deg \; \mathcal{L}_Z = m e + (m-1)^2 e  + (m-1) e = m^2e.$$
Hence, $\mathcal{L}_Z$ is a non--special line bundle on $Z$, $Z$ being a reduced, connected and nodal curve of arithmetic genus $p_a(Z) = g = g_m$ as in \eqref{genusg} (the non--speciality of 
$\mathcal{L}_Z$ can be proved by applying the same numerical computation as in the proof of
Proposition \ref{prop:Flam7e} (v), replacing the canonical bundle with the dualizing sheaf $\omega_Z$).
Thus, from \eqref{eq:Flam**} one gets
$$h^0(Z, N_{Z/\mathbb{P}^{R}}) =  h^0(Z, \psi^*(N_{Y/\Pp^{R-1}})) + h^0(Z, \mathcal{L}_Z)$$where
$$h^0(Z, \mathcal{L}_Z) = \chi(Z, \mathcal{L}_Z) = \frac{m(m+1)}{2} e - m(\gamma-1),$$both equality following from the non--speciality of $\mathcal{L}_Z$.
As for the summand $h^0(Z, \psi^*(N_{Y/\Pp^{R-1}}))$, we can apply
projection formula and Leray's isomorphism, which gives
$$h^0(Z, \psi^*(N_{Y/\Pp^{R-1}})) = h^0(Y, N_{Y/\Pp^{R-1}}) +   h^0(Y, N_{Y/\Pp^{R-1}} \otimes \mathcal{T}_{m}^{\vee}).$$Since
$h^0(Y, N_{Y/\mathbb{P}^{R-1}}) = \lambda_{e,\gamma,R-1}$ as $[Y] \in \widehat{\mathcal{I}_{e,\gamma,R-1}}$ is general (cf. Corollary \ref{cor:Sernesi}), then
comparing with \eqref{eq:chestrazio} we deduce that $h^0(Y, N_{Y/\mathbb{P}^{R-1}} \otimes \mathcal{T}_{m}^{\vee}) \geqslant R$.
Thus, using the previous inequality \eqref{eq:ineq}, we get $h^0(Y, N_{Y/\mathbb{P}^{R-1}}\otimes\mathcal{T}_{m}^{\vee}) =R$.

By semi--continuity on the general element $[X_m] \in  {\mathcal S}_{d,g,R}$, with its degree--$m$ covering map
$\varphi_m: X_m \to Y$  and its associated Tschirnhausen bundle $\mathcal{T}_{\varphi_m}^{\vee}$, we deduce that
$$h^0(Y, N_{Y/\mathbb{P}^{R-1}}\otimes \mathcal{T}_{\varphi_m}^{\vee}) \leqslant R.$$On the other hand, replacing $Z$ with $X_m$ in the previous computations, since 
$$h^0(X_m, N_{X_m /\mathbb{P}^{R}}) = \dim\; T_{[X_m]} (\mathcal I_{d,g,R}) \geqslant \dim\; {\mathcal S}_{d,g,R} =  \lambda_{e,\gamma,R-1} \, + R + \frac{m(m+1)}{2}e -m(\gamma-1),$$
one can conclude  by applying \eqref{eq:Flam***}, with ${\mathcal O}_{X_m} (R_{\varphi_m}) \cong {\mathcal O}_{X_m}(m-1)$ as in \eqref{eq:ramifdiv}, and reasoning as we did for $Z$ above. 
\end{proof}

The previous computations show that, for $[X] \in  {\mathcal S}_{d,g,R}$ general, one has 
\begin{equation}\label{eq:tangentSdgR}
\dim {\mathcal S}_{d,g,R} = \lambda_{d,g,R} + \sigma_{d,g,R} = \dim T_{[X]} ({\mathcal S}_{d,g,R}) = T_{[X]} (\mathcal I_{d,g,R}),
\end{equation} which therefore implies that ${\mathcal S}_{d,g,R}$ is generically smooth. 
\end{proof}

We are finally in position to prove our Main Theorem.

\begin{proof} [Proof of Main Theorem] The first part of Lemma \ref{lem:SdgR} ensures that any irreducible component of $\mathcal I_{d,g,R}$ containing 
${\mathcal S}_{d,g,R}$ has to be {\em superabundant}, having dimension at least $\dim {\mathcal S}_{d,g,R} = \lambda_{d,g,R} + \sigma_{d,g,R}$. On the other hand, 
the proofs of Claims \ref{cl:calcoloparz} and \ref{cl:induction} show that ${\mathcal S}_{d,g,R}$ is contained 
in a unique component of $\mathcal I_{d,g,R}$, more precisely it fills--up an open, dense subset
of an irreducible component of $\mathcal I_{d,g,R}$ which is generically smooth, superabundant, of dimension
$\lambda_{d,g,R} + \sigma_{d,g,R}$. Indeed, by \eqref{eq:tangentSdgR}, for $[X] \in {\mathcal S}_{d,g,R}$ general we have that 
$$\dim T_{[X]} ({\mathcal S}_{d,g,R}) = h^0(X, N_{X/\Pp^R}) = \dim T_{[X]} ({\mathcal I}_{d,g,R}) = \dim {\mathcal S}_{d,g,R} = \lambda_{d,g,R} + \sigma_{d,g,R}.$$
\end{proof}

\begin{remark} {\normalfont It is clear from the construction that $\mathcal{S}_{d, g, R}$ lies in
a component of $\mathcal{I}_{d, g, R}$ which cannot dominate $\mathcal M_g$. Indeed, the modular morphism of such a component maps to the {\em Hurwitz space} $\mathcal{H}_{\gamma,m,g}$ parametrizing isomorphism classes of genus--$g$ curves arising as $m$-sheeted, ramified covers of irrational curves of
genus $\gamma$.
}
\end{remark}


\end{document}